\documentclass[11pt]{amsart}
\usepackage[margin=1in]{geometry}
\usepackage{graphicx,amsfonts, amsmath, amssymb, amsthm, enumerate, xcolor,enumitem}
\usepackage[T1]{fontenc}
\usepackage[indent]{parskip}
\usepackage{hyperref}
\usepackage{cleveref}
\usepackage{tikz}

\setlist[enumerate]{topsep=1ex,itemsep=0ex,partopsep=1ex,parsep=1ex}
\def\norm#1{\left\|#1\right\|}

\newtheorem{thm}{Theorem}[section]
\newtheorem{lem}[thm]{Lemma}

\theoremstyle{definition}
\newtheorem{dfn}[thm]{Definition}

\newtheorem*{prb*}{Problem}
\newtheorem*{thm*}{Theorem}
\newtheorem{ex}[thm]{Example}
\theoremstyle{remark}
\newtheorem{rmk}[thm]{Remark}
\newcommand{\Acover}{A_{\mathrm{cover}}}
\newcommand{\alpsmall}{\alpha_{1}}
\newcommand{\Asmall}{A_{\mathrm{small}}}
\newcommand{\alpcover}{\alpha_{2}}
\newcommand{\Tsmall}{\Theta_{\mathrm{small}}}
\newcommand{\Tcover}{\Theta_{\mathrm{cover}}}
\newcommand{\thetamin}{\theta_{\mathrm{min}}}
\newcommand{\thetamax}{\theta_{\mathrm{max}}}
\newcommand{\thetatot}{\theta_{\mathrm{tot}}}

\newcommand{\aminus}{\alpha^-}
\newcommand{\aplus}{\alpha^+}

\newcommand{\newA}{F}
\DeclareMathOperator{\inter}{int}
\DeclareMathOperator{\dist}{dist}

\providecommand{\al}{\alpha}
\providecommand{\be}{\beta}
\providecommand{\de}{\delta}
\providecommand{\De}{\Delta}

\providecommand{\Ga}{\Gamma}
\providecommand{\e}{\epsilon}
\providecommand{\emp}{\varnothing}

\providecommand{\R}{\mathbb{R}}

\providecommand{\Z}{\mathbb{Z}}
\providecommand{\N}{\mathbb{N}}

\providecommand{\T}{\mathbb{T}}

\newcommand{\cA}{\mathcal{A}}
\newcommand{\cB}{\mathcal{B}}

\newcommand{\cF}{\mathcal{F}}

\newcommand{\cL}{\mathcal{L}}

\newcommand{\nbhd}[2]{\operatorname{nbhd}_{#2}(#1)}

\graphicspath{{figures/}}

\numberwithin{equation}{section}

\title[Generalized prescribed projections and efficient covering] {Prescribed projections for a general class of maps and efficient covering by variable plane curves}

\author{Alan Chang}
\address{Department of Mathematics, Washington University in St. Louis, St. Louis, MO}
\email{alanchang@math.wustl.edu}

\author{Robert Fraser}
\address{Department of Mathematics, Wichita State University, Wichita, KS}
\email{robert.fraser@wichita.edu}

\author{Alex McDonald}
\address{Department of Mathematics, Kennesaw State University, Marietta, GA}
\email{amcdon79@kennesaw.edu}

\author{Mark Meyer}
\address{LAMA, Univ Gustave Eiffel, Univ Paris Est Creteil, 77447 Marne-la-Vall´ee, France.
}
\email{mark.meyer@univ-eiffel.fr}

\author{Krystal Taylor}
\address{Department of Mathematics, The Ohio State University, Columbus, OH}
\email{taylor.2952@osu.edu}

\thanks{
Chang was partially supported by NSF grant DMS-2247233.  
Fraser: This material is based upon work supported by the National Science Foundation under Award No. DMS-2453364.  
McDonald is supported by an AMS-Simons
Research Enhancement Grant for Primarily Undergraduate Institution Faculty.  
Meyer is supported by the NSF MSPRF grant number 2502794.  
Taylor is supported in part by the Simons Foundation Grant GR137264. 
Part of this work was done at the 48th Summer Symposium in Real Analysis at Washington University in St. Louis.
The symposium was funded by NSF grant DMS-2606135.}

\begin{document}
\begin{abstract}
We prove a prescribed projection theorem for any class of non-linear projections defined on open subsets of the plane which satisfies some mild geometric conditions.  This generalizes Davies's classic result on prescribed projections for orthogonal projections, as well as a recent result of authors, Chang, McDonald, and Taylor, 
on a specific class of non-linear projection maps.  We also give several geometric applications, including a proof that certain classes of plane curves can be used to cover arbitrary measurable sets with measure zero error.
\end{abstract}
\maketitle

\section{Introduction}
A major theme in the intersection of harmonic analysis and geometric measure theory is the use of measure theoretic properties of projections to study the geometry of a set.  A foundational result in this area is \textit{Marstrand's Theorem}.  Let $\pi_\theta:\R^2\to \R$ denote orthogonal projection in the direction $\theta$, and let $E\subset \R^2$ have Hausdorff dimension $s$.   Marstrand \cite{MR63439} proved that if $s>1$ then $\pi_\theta(E)$ has positive measure for almost every $\theta$, and if $s\leq 1$ then $\pi_\theta(E)$ has dimension $s$ for almost every $\theta$.  In particular, if $s\neq 1$, then there is a full-measure set $\Theta$ of directions such that either $|\pi_\theta(E)|>0$ for all $\theta\in \Theta$, or $|\pi_\theta(E)|=0$ for all $\theta\in \Theta$ (here and throughout, we use vertical bars $|\cdot|$ to denote the Lebesgue measure of a set).  When $s=1$, however, the situation is considerably more subtle.  Marstrand's theorem guarantees that $\pi_\theta(E)$ has dimension 1 for almost every $\theta$, but does not allow one to distinguish between a one-dimensional set with measure zero and one with positive measure.  A remarkable result due to Davies \cite{Davies52} (later generalized to higher dimensions by Falconer \cite{Falconer86}) shows that, in fact, information about the projection in one direction may tell us very little about projections in other directions.  In particular, Davies showed that for any family $\{A_\theta\}_{0\leq \theta<\pi}$ of subsets of $\R$ satisfying a mild measurability condition, there exists $E\subset \R^2$ such that for almost every $\theta$, we have $\pi_\theta(E)\supset A_\theta$ and $|\pi_\theta(E)\setminus A|=0$.  In other words, it is possible to choose targets for the projections $\pi_\theta(E)$ essentially independently, and find a single set $E$ which hits those targets up to measure zero error.  In particular, we may choose $A_\theta$ to have measure zero for many directions and to have positive measure for many other directions.  

We refer to results like Davies's as \textit{prescribed projection theorems}.  Such theorems are closely connected to natural geometric questions.  By point-line duality, Davies's theorem is readily seen to be equivalent to the following result: For any measurable $F\subset \R^2$, there exists a set $E\subset \R^2$ which is a union of lines and satisfies $E\supset F,|E\setminus F|=0$.  We call such results \textit{efficient covering theorems}.  

Efficient covering theorems are closely connected to other problems, such as Kakeya and Nikodym problems, which require the construction of sets which have many lines or line segments but small measure. 
For a thorough exposition of the relationship between 
prescribed projection results and the study of Nikodym sets, see \cite[Chapter 7]{FalconerGFS}).  

In recent years, there has been considerable interest in studying classical projection theorems for non-linear projections \cite{BV11, BT23, CDT22, DT22, HJJL12, LMT26, PS00}, as well as related non-linear analogues of Kakeya/Nikodym-type constructions \cite{CC19, CG24, CGY25, Zah23}.  
Authors Chang, McDonald, and Taylor  \cite{CMT23} proved a prescribed projection result for the family of curve projection maps 
\[
\Phi_\al(x_1,x_2)=f(x_1-\al)+x_2,
\]
where $f$ is a $C^1$ function on a compact interval with monotone derivative and $\alpha$ is an index for the family of maps.  This was motivated by the following efficient covering problem: \textit{What curves $\Ga$ have the property that an arbitrary measurable $F\subset \R^2$ can be efficiently covered by translates $(\Ga+x)$?}  The authors showed that it is sufficient for $\Ga$ to be the graph of a function $f$ with the aforementioned properties.

The goal of this paper is to generalize the non-linear prescribed projection theorem from \cite{CMT23} to a much broader class of projections, which can be defined on a broader class of domains.  This presents many obstacles, both of a technical and conceptual nature, and requires several new ideas.  In Section \ref{applications}, we provide several applications to geometric problems.

\subsection{Definitions and main theorem}
\label{assumptions}
Throughout, $D\subset \R^3$ is a fixed open set, and $\phi:D\to\R$ is a fixed function.  We will view $\R^3$ as $\R\times \R^2$, and write elements as $(\al,x)$ where $\al\in\R$ and $x=(x_1,x_2)\in \R^2$.  For fixed $\al$, we let $D_\al=\{x:(\al,x)\in D\}$, and we define $\cA=\{\al:D_\al\neq\emp\}$.  Finally, for $\al\in \cA$, we define the generalized projection $\phi_\al:D_\al\to \R$ by $\phi_\al(x)=\phi(\al,x)$. 

We write $\nabla\phi_\al(x) = \nabla_x \phi(\alpha,x)$ for the gradient in the $x$-variable only. We say $\phi$ is \textbf{non-stationary} if $\nabla\phi_\al(x) \neq 0$ for all $(\al,x)\in D$.

By \textbf{direction} of a vector $v\in \R^2\setminus\{0\}$, we mean the equivalence class of $\frac{v}{|v|}$ under the equivalence relation defined by identifying antipodal points of the unit circle.  Thus, a direction can be represented by a real number modulo $\pi$.  Given such a representation and a non-stationary differentiable function $\phi$, we let $\theta_\al(x)$ denote the number representing the direction perpendicular to $\nabla \phi_\al(x)$; that is, $\theta_\al(x)$ is the direction of the tangent to the level curve $\phi_\al^{-1}(\phi_\al(x))$ at the point $x$.

\begin{dfn}[Angle monotonicity condition]
\label{dfn:angle monotonicity}
We say that a non-stationary $C^1$ function $\phi : D \to \R$ satisfies the \textbf{angle monotonicity condition} if for each fixed $x$, there is a choice of representatives $[a, a+\pi)$ for $\R/\pi\Z$ such that the function $\alpha \mapsto \theta_{\alpha}(x)$ is a strictly increasing function $\{\alpha : (\alpha,x) \in D\} \to [a,a+\pi)$.
\end{dfn}

With these definitions in place, we are ready to state our main theorem.

\begin{thm}[Generalized prescribed projection theorem]
\label{thm:main}
Let $D \subset \R^3$ be an open set and let $\phi : D \to \R$ be a $C^1$ function such that:
\begin{enumerate}
\item \label{cond nsam} $\phi$ is non-stationary and satisfies the angle monotonicity condition.
\item \label{cond ext} $\phi$ and $\nabla_x\phi$ both extend continuously to the closure $\overline D$, and $\nabla_x \phi$ is nonzero on $\overline D$. 
\item \label{cond conv} $D_\alpha$ is convex for all $\alpha \in \R$.
\end{enumerate}
Suppose $\{F_\al\}_{\al\in\R}$ is a family of subsets of $\R$ such that $F_\al\subset \phi_\al(D_\al)$ for each $\al$, and
$
\bigcup_{\alpha\in\R} (\{\alpha\}\times F_\alpha)
$
is a measurable subset of $\R^2$.  Then there exists a Borel set $E\subset \R^2$ such that
\begin{align*}
&\phi_\alpha(E \cap D_\alpha) \supset F_\alpha &&\text{for all } \alpha \in \R
\\
&|\phi_\alpha(E \cap D_\alpha)\setminus F_\alpha|=0
&&\text{for almost every } \alpha \in \R.
\end{align*}
\end{thm}

\begin{rmk}
We can replace conditions (2) and (3) in \Cref{thm:main} with a slightly weaker condition; see \Cref{rmk:rectangle squishing condition} and \Cref{thm:main v2}. 
\end{rmk}

\subsection{Notation and conventions}

Given $\de>0$ and $E\subset \R^2$, we write $\nbhd{E}{\de}$ for the open $\de$-neighborhood of $E$.  We denote the open disk of radius $r>0$ centered at $x\in \R^2$ by $B(x,r)$, and we denote the closed disk by $\overline{B}(x,r)$.  

We write $X\lesssim Y$ to mean there exists a constant $C>0$ such that $X\leq CY$.  We write $X\gtrsim Y$ to mean $Y\lesssim X$, and $X\approx Y$ to mean both $X\lesssim Y$ and $Y\lesssim X$ hold.  Throughout, we will consider the function $\phi$ and domain $D$ to be fixed; thus, constants will always be allowed to depend on these quantities.  When more parameters arise, we will always specify which ones the implied constants may depend on.

\subsection{Acknowledgments}

We thank Shaoming Guo and Joshua Zahl for helpful discussions. AI tools were used for proofreading and generating figures.

\section{Preliminary geometric lemmas}

\begin{dfn}[$\R/\pi\Z$]
We identify $\R/\pi\Z$ with the space of unoriented directions in $\R^2$, via $\theta\mapsto \pm (\cos\theta,\sin\theta)$.

\begin{enumerate}
\item Cyclic ordering: For $\theta_1, \ldots, \theta_n \in \R/\pi\Z$, we write 
\[
\theta_1 < \theta_2 < \cdots < \theta_n
\]
to mean that if we start at $\theta_1$ and move counter-clockwise to $\theta_n$, we pass through $\theta_1, \ldots, \theta_{n}$ in that order. Equivalently, there exist representatives $\theta_1', \ldots, \theta_n' \in \R$ such that $\theta_1' < \cdots < \theta_n'$, and $\theta_n' - \theta_1' < \pi$. We make a similar definition for
\[
\theta_1 \leq \theta_2 \leq \cdots \leq \theta_n.
\]

\item Intervals/arcs: For $\theta_1, \theta_2 \in \R/\pi\Z$, we define the interval 
\[
[\theta_1, \theta_2] :=
\{\theta \in \R/\pi\Z : \theta_1 \leq \theta \leq \theta_2\}
.
\]
(We use the definition of ``$\leq$'' from above.) The length of this interval is
\[
\theta_2 - \theta_1
:=
\min\{d \geq 0 : \theta_1 + d \equiv \theta_2 \bmod{\pi}
\}
\in [0,\pi)
.
\]

\item Distance: For $\theta_1, \theta_2 \in \R/\pi\Z$, the distance between $\theta_1$ and $\theta_2$ is defined to be
\[
|\theta_1 - \theta_2| 
:= 
\min\{\theta_1-\theta_2, \theta_2-\theta_1\}.%
\] 
\end{enumerate}
\end{dfn}

In this terminology, the angle monotonicity condition can be rephrased as saying that for any choice of $\alpha_1 < \cdots < \alpha_n$ for which $\theta_{\alpha_i}(x)$ is defined, we have $\theta_{\alpha_1}(x) < \cdots < \theta_{\alpha_n}(x)$.

\subsection{Projections of thin rectangles}

We introduce a geometric condition. To motivate it, observe the following elementary fact: fix a small $\epsilon > 0$, and consider the thin vertical rectangle $[0, \epsilon \ell] \times [0, \ell]$. If we project this onto any line that makes angle at most $\epsilon$ with the horizontal axis, then the image has size comparable to $\epsilon \ell$. Thus, projection in these directions ``compresses'' the rectangle by a factor of $\epsilon$. We need this to hold for our nonlinear maps $\phi_\alpha$, at least for small scales.

\begin{dfn}[Rectangle squishing condition]
\label{dfn:rectangle squishing condition}
We say $\phi : D \to \mathbb{R}$ satisfies the \textbf{rectangle squishing condition} if for every bounded subset $D' \subset D$ and every $\epsilon > 0$, there exists $\delta = \delta(D', \epsilon) > 0$ such that the following holds: Let $L \subset \mathbb{R}^2$ be a line segment of length $|L| \le \delta$. Let $\theta_L$ denote the direction of $L$, and suppose $\alpha \in \R$ is such that there exists $x \in D'_\alpha \cap \nbhd{L}{\delta|L|}$ satisfying $|\theta_\alpha(x) - \theta_L| \le \delta$. Then $\phi_\alpha(\nbhd{L}{\delta|L|} \cap D_\alpha)$ is contained in an interval of length at most $\epsilon |L|$.
\end{dfn}

The next lemma shows that the rectangle squishing condition holds under the assumptions of \Cref{thm:main}.

\begin{lem}
\label{lem:rectangle squishing condition}
Suppose $\phi$ and $\nabla_x \phi$ extend continuously to the closure $\overline{D}$, and that the extension of $\nabla_x \phi$ is nonzero on $\overline{D}$. Suppose also that $D_\alpha$ is convex for all $\alpha \in \R$. Then $\phi$ satisfies the rectangle squishing condition.
\end{lem}

\begin{proof}
Let $D' \subset D$ be bounded and $\epsilon > 0$. Consider the compact set $K = \overline{\nbhd{D'}{1} \cap D}$. Since $\theta_\alpha(x)$ is uniformly continuous on $K$, we can choose $\delta \in (0,\min(1/3,\epsilon))$ such that $|x - y| \le 3\delta$ implies $|\theta_\alpha(x) - \theta_\alpha(y)| \le \epsilon$ on $K$.

Let $L \subset \mathbb{R}^2$ be a line segment with $|L| \le \delta$, and let $x \in D'_\alpha \cap \nbhd{L}{\delta|L|}$ satisfy $|\theta_\alpha(x) - \theta_L| \le \delta$. We need to show
\[
\sup_{y,z \in \nbhd{L}{\delta|L|} \cap D_\alpha}
|\phi_\alpha(y)-\phi_\alpha(z)|
\leq
\e |L|.
\]
For any $y,z \in \nbhd{L}{\delta|L|} \cap D_\alpha$, convexity and the mean value theorem implies there exists $w$ on the line segment joining $y$ and $z$ such that
\[
|\phi_\alpha(y) - \phi_\alpha(z)| = |\nabla \phi_\alpha(w) \cdot (y - z)| = \|\nabla \phi_\alpha(w)\| \|P(y-z)\|,
\]
where $P : \R^2 \to \R^2$ denotes the orthogonal projection onto the span of $\nabla \phi_\alpha(w)$. Since $y,z \in \nbhd{L}{\delta|L|}$, an elementary geometric argument gives
\[
\|P(y-z)\|
\leq
2\delta|L| + |P(L)|
=
(2\delta + \sin|\theta_\alpha(w) - \theta_L|)|L|.
\]
Since $|x - w| \le \operatorname{diam}(\nbhd{L}{\delta|L|}) \leq (1+2\delta)|L| \leq 3\delta < 1$, we have $(\alpha,x),(\alpha,w) \in K$. Thus, uniform continuity gives $|\theta_\alpha(w) - \theta_L| \le |\theta_\alpha(w) - \theta_\alpha(x)| + |\theta_\alpha(x) - \theta_L| \le \epsilon + \delta \leq 2\epsilon$. Thus, $\|P(y-z)\| \leq 4\epsilon$, so
\[
|\phi_\alpha(y) - \phi_\alpha(z)| \leq 4M \epsilon |L|, \qquad\text{where } M := \sup_{(\alpha,x) \in K} \|\nabla_x \phi(\alpha,x)\| < \infty.
\]
This shows that $\phi_\alpha(\nbhd{L}{\delta|L|} \cap D_\alpha)$ lies in an interval of length at most $4M \epsilon |L|$. Since $M$ does not depend on $\epsilon$, we can replace $\epsilon$ with $\epsilon/(4M)$ at the beginning of the proof.
\end{proof}

It will be convenient to have a shorthand for the class of functions we are considering.

\begin{dfn}
Let $D \subset \R^3$ be an open set. We say a function $\phi:D\to \R$ is \textbf{admissible} if $\phi$ is $C^1$ and non-stationary (i.e., $\nabla_x\phi$ is nonvanishing on $D$), satisfies the angle monotonicity condition (\Cref{dfn:angle monotonicity}), and the rectangle squishing condition (\Cref{dfn:rectangle squishing condition}).
\end{dfn}

\begin{rmk}
\label{rmk:rectangle squishing condition}
\Cref{lem:rectangle squishing condition} is the only way that the proof of \Cref{thm:main} uses conditions (2) and (3). Thus, we can replace conditions (1), (2), and (3) of \Cref{thm:main} with ``$\phi$ is admissible'' to get a slightly more general theorem; see \Cref{thm:main v2}.
\end{rmk}

\subsection{Angle separation}

\begin{lem}[Angle separation]\label{thetaseplem}
Let $\phi$ be admissible, let $\Asmall$ and $\Acover$ be disjoint compact sets, and assume $\Acover$ is an interval. Let $K$ be a compact set contained in $\bigcap_{\alpha \in \Acover} D_\alpha$. Then there exist $\sigma_0(K, \Asmall, \Acover)$ and $\delta_0(K, \Asmall, \Acover)> 0$ such that if $x_1, x_2$ are points in $K$, $\alpsmall \in \Asmall$, and $\alpcover \in \Acover$, and $x_1 \in D_{\alpsmall}$, and $|x_1 - x_2| < \delta_0$, then $|\theta_{\alpcover}(x_2) - \theta_{\alpsmall}(x_1)| > \sigma_0$.
\end{lem}

\begin{proof}
By assumption, $\Acover \times K \subset D$. Since $\Acover \times K$ is compact, $D$ is open, and $\Acover \cap \Asmall = \emp$, there exists an interval $[\aminus, \aplus]$ such that (Figure \ref{fig: As and alphas})
\begin{gather}
\label{eq:aminus aplus Acover}
(\aminus, \aplus) \supset \Acover
\\
[\aminus, \aplus] \times K \subset D
\\
\label{eq:disjoint Asmall Acover nbhd}
[\aminus, \aplus] \cap \Asmall = \emptyset
\end{gather}

\begin{figure}
\centering
\begin{minipage}{.45\textwidth}
  \centering
\begin{tikzpicture}
\draw(0,0)--(6,0);
\draw[ultra thick, blue](2,0)--(4,0);
\draw[ultra thick, red](.5,0)--(1.5,0);
\draw[ultra thick, red](4.5,0)--(5.5,0);

\draw (1.2,-.1)--(1.2,.1);
\draw (1.8,-.1)--(1.8,.1);
\draw (4.1,-.1)--(4.1,.1);
\draw (2.4,-.1)--(2.4,.1);

\node[below left] at (1.2,-.1) {$\alpsmall$};
\node[below] at (1.8,-.1) {$\al^-$};
\node[below right] at (2.4,-.1) {$\alpcover$};
\node[below right] at (4.1,-.1) {$\al^+$};

\node[red] at (1,1) {$\Asmall$};
\node[blue] at (3,1) {$\Acover$};
\node[red] at (5,1) {$\Asmall$};
\end{tikzpicture} 
\caption{The sets $\Acover,\Asmall$.}
  \label{fig: As and alphas}
\end{minipage}%
\begin{minipage}{.45\textwidth}
  \centering
\begin{tikzpicture}
\draw (2.5,0) arc (0:180:2.5);
\draw[ultra thick, blue] (0,0) +(75:2.5) arc (75:103:2.5);
\draw[ultra thick, red] (0,0) +(10:2.5) arc (10:50:2.5);
\draw[ultra thick, red] (0,0) +(120:2.5) arc (120:170:2.5);
\draw[fill](2.5*.9,2.5*.44)circle[radius=0.1];
\draw[fill](2.5*.4,2.5*.92)circle[radius=0.1];
\draw[fill](-2.5*.1,2.5*1)circle[radius=0.1];
\draw[fill](-2.5*.3,2.5*.95)circle[radius=0.1];

\node[below right] at (.2+2.5*.9,2.5*.44) {$\theta_{\alpsmall}(x)$};
\node[right] at (.2+2.5*.4,2.5*.92) {$\theta_{\aminus}(x)$};
\node[above right] at (-2.5*.1,2.5*1) {$\theta_{\alpcover}(x)$};
\node[above left] at (-2.5*.3,2.5*.95) {$\theta_{\aplus}(x)$};
\end{tikzpicture}
\caption{Angles on the half-circle, with sets $\Acover$ (blue) and $\Asmall$ (red) for emphasis.}
  \label{fig: circle with angles}
\end{minipage}
\end{figure}
The angle monotonicity condition implies that 
\begin{align*}
    \theta_{\aminus}(x) < \theta_{\alpha}(x) < \theta_{\aplus}(x) 
    \qquad\text{for all } (\alpha,x) \in \Acover \times K.
\end{align*} 
Since $\theta_\alpha(x)$ is continuous and $\Acover \times K$ is compact, there exists $\sigma_0 = \sigma_0(K, \Asmall, \Acover)$ such that 
\begin{align}
\label{eq:2sigma separation}
    \begin{cases}
    \theta_{\alpha}(x) - \theta_{\aminus}(x) > 2\sigma_0 \\ \theta_{\aplus}(x) - \theta_{\alpha}(x) > 2\sigma_0
    \end{cases}
    \qquad\text{for all } (\alpha,x) \in \Acover \times K.
\end{align}

Now, let $\alpsmall \in \Asmall$ and $\alpcover \in \Acover$, and suppose $x \in D_{\alpsmall} \cap K.$ By \eqref{eq:aminus aplus Acover} and \eqref{eq:disjoint Asmall Acover nbhd}, the angle monotonicity condition implies
\begin{align}
    \theta_{\alpsmall}(x) < \theta_{\aminus}(x) <  \theta_{\alpcover}(x) < \theta_{\aplus}(x),
\end{align}
as in Figure \ref{fig: circle with angles}.  In other words, the angles $\theta_{\alpsmall}(x)$ and $\theta_{\alpcover}(x)$ lie in the complementary arcs $[\theta_{\aplus}(x),\theta_{\aminus}(x)]$ and $[\theta_{\aminus}(x),\theta_{\aplus}(x)]$, respectively. Furthermore, by \eqref{eq:2sigma separation}, $\theta_{\alpcover}(x) \in [\theta_{\aminus}(x)+2\sigma_0,\theta_{\aplus}(x)-2\sigma_0]$. This implies 
\begin{align}
    |\theta_{\alpcover}(x) - \theta_{\alpsmall}(x)| > 2 \sigma_0 \qquad \text{whenever $\alpcover \in \Acover$, $\alpsmall \in \Asmall$, $x \in D_{\alpsmall} \cap K$.}
\end{align}

Next, because $\theta_{\alpha}(x)$ is uniformly continuous on $\Acover \times K$, there exists $\delta_0 := \delta_0(K, \Asmall, \Acover)$ such that 
\begin{align}
    |\theta_{\alpha}(y) - \theta_{\alpha}(z)| < \sigma_0
    \qquad\text{whenever $\alpha \in \Acover$ and $y, z \in K$ satisfy $|y - z| < \delta_0$}.
\end{align}  

The lemma follows from the two inequalities above and the reverse triangle inequality.
\end{proof}

\subsection{Intersections between line segments and fibers of $\phi_\alpha(x)$}

We conclude the section with a key geometric lemma (\Cref{geomlem phi}), which we will use to ensure that the fibers of our projections intersect certain line segments.

\begin{figure}
\centering
\includegraphics[page=1,width=0.4\textwidth]{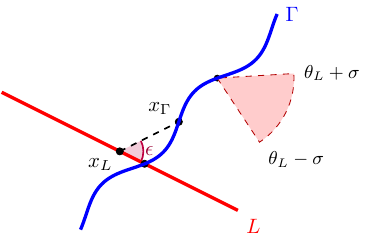}
\caption{Illustration of \Cref{geomlem}}
\label{fig: geo lemma 1}
\end{figure}

\begin{lem}[Intersections between segments and curves]\label{geomlem}
Let $L \subset \R^2$ be a line segment with center $x_L$, direction $\theta_L$, and length $2\ell_L$. Let $\Gamma \subset \R^2$ be a $C^1$ curve containing a point $x_\Gamma$ such that:
\begin{enumerate}
    \item The two connected components of $\Gamma \setminus \{x_\Gamma\}$ each have arc length at least $\ell_\Gamma$.
    \item The tangent direction at every point of $\Gamma$ lies in the complement of $(\theta_L - \sigma, \theta_L + \sigma)$ for some $\sigma \in (0, \frac{\pi}{2})$.
\end{enumerate}
Let $\epsilon \in (0, \frac{\pi}{2}]$ be the angle between the vector $x_\Gamma - x_L$ and the line segment $L$. If 
\begin{align*}
\ell_L 
\geq 
\frac{\sin(\sigma+\epsilon)}{\sin \sigma}
\|x_L-x_\Gamma\| 
\qquad\text{and}\qquad
\ell_\Gamma
\geq 
\frac{\sin\epsilon}{\sin \sigma}
\|x_L-x_\Gamma\| 
,
\end{align*}
then $\Gamma$ and $L$ must intersect at some point (Figure \ref{fig: geo lemma 1}).
\end{lem}

\begin{proof}
The statement is invariant under translations, rotations, and reflections. Thus, we may without loss of generality assume:
\begin{enumerate}
    \item $x_L = (0,0)$ and $L$ is a subset of the vertical axis $\{0\} \times \R$.
    \item $x_\Gamma = \|x_L - x_\Gamma\|(-\sin \epsilon, \cos\epsilon) =: (a,b)$.
    \item $\Gamma$ is the graph of a $C^1$ function $f : [p,q] \to \R$ with $|f'| \leq \cot\sigma$
\end{enumerate}

First, $\Gamma$ intersects the vertical axis $\{0\} \times \R$. To see this, note that the assumptions on $\Gamma$ give
\[
\ell_\Gamma
\leq
\int_{a}^q \, \sqrt{1+f'(t)^2} \, dt
\leq
\int_{a}^q \, \sqrt{1+\cot^2\sigma} \, dt
=
q \csc \sigma
+
\|x_L - x_\Gamma\| \sin \epsilon \csc \sigma
\leq
q \csc\sigma+ \ell_\Gamma.
\]
Thus, $q \geq 0$, so $\Gamma$ intersects the vertical axis $\{0\} \times \R$.

To complete the proof, we need to show that $|f(0)| \leq \ell_L$. By the mean value theorem,
$|f(0) - b|
=
|f(0) - f(a)|
\leq
|a| \cot\sigma$. Then by the triangle inequality and the trigonometric identity for $\sin(\sigma+\epsilon)$, we have
\[
|f(0)|
\leq
|b| + |a| \cot \sigma
=
\|x_L-x_\Gamma\|\sin(\sigma+\epsilon) \csc\sigma
\leq
\ell_{L}
\]
as desired.
\end{proof}

\begin{lem}[Intersections between segments and fibers of $\phi_\alpha$]
\label{geomlem phi}
Let $\phi : D \to \R$ be admissible. Let $L$ be a line segment with center $x_L$, direction $\theta_L$, and length $2\ell_L$. Fix $\al_0$ and a ball $\overline B(x_0,r) \subset D_{\al_0}$. Suppose that $|\theta_{\al_0}(x) - \theta_L| \geq \sigma$ for all $x \in \overline B(x_0, r)$. Let $\epsilon \in (0, \frac{\pi}{2}]$ be the angle between the vector $x_0 - x_L$ and the line segment $L$. If
\begin{align}
\label{eq:geomlem length condition}
\ell_L 
\geq 
\frac{\sin(\sigma+\epsilon)}{\sin \sigma}
\|x_L-x_0\| 
\qquad\text{and}\qquad
r
\geq 
\frac{\sin\epsilon}{\sin \sigma}
\|x_L-x_0\| 
\end{align}
then $\phi_{\al_0}(x_0) \in \phi_{\al_0}(L)$ (or equivalently, the fiber of $\phi_{\al_0}$ through $x_0$ intersects $L$).
\end{lem}

\begin{proof}
Let $\Gamma'$ denote the connected component of the level set of $\phi_{\alpha_0}$ containing $x_0$. By nonvanishing of $\nabla \phi_{\alpha_0}$ and the implicit function theorem, $\Gamma'$ is a closed $1$-dimensional submanifold of $D_{\alpha_0}$, so it is either diffeomorphic to $\R$ or to $S^1$. If $\Gamma'$ is diffeomorphic to $S^1$, then  $\phi_{\alpha_0}$ has an absolute maximum or minimum enclosed inside $\Gamma'$, which contradicts the nonvanishing of $\nabla \phi_{\alpha_0}$. Since $\Gamma'$ is diffeomorphic to $\R$, both connected components of $\Gamma'\setminus\{x_0\}$ must leave the compact subset $\overline B(x_0, r) \subset D_{\alpha_0}$.

Let $\Gamma$ be the connected component of $\Gamma' \cap \overline B(x_0, r)$ containing $x_0$. Then we can apply \Cref{geomlem} to conclude that $L$ and $\Gamma$ intersect at some point, say $x_1$. We have $\phi_{\alpha_0}(x_0) = \phi_{\alpha_0}(x_1) \in \phi_{\alpha_0}(L)$, which completes the proof.
\end{proof}

\begin{rmk}
In \Cref{sec: Venetian Blinds}, we will apply \Cref{geomlem phi} twice. In the first application, we assume
\begin{equation}\label{eq:geomlem 1}
\min\{\ell_L, r\} \geq \frac{2}{\sigma}\|x_L - x_0\|.
\end{equation}
In the second application, we assume
\begin{equation}\label{eq:geomlem 2}
0 \le \epsilon < \sigma < \frac{\pi}{2}, \qquad \ell_L \geq \left(1+\frac{\epsilon}{\sigma}\right) \|x_L-x_0\|, \qquad r \geq \frac{2\epsilon}{\sigma} \|x_L-x_0\|.
\end{equation}
Each of \eqref{eq:geomlem 1} and \eqref{eq:geomlem 2}, considered separately, implies \eqref{eq:geomlem length condition} by the elementary inequalities $\sin \sigma \geq \frac{\sigma}{2}$ for $\sigma \in (0, \frac{\pi}{2})$, $\sin(x) \le 1$, and $\frac{\sin(\sigma+\epsilon)}{\sin\sigma} = \cos\epsilon + \frac{\sin\epsilon}{\tan\sigma} \le 1 + \frac{\epsilon}{\sigma}$.
\end{rmk}

\section{The Venetian blinds construction}
\label{sec: Venetian Blinds}

The main goal of this section is to prove the following.

\begin{thm}[Iterated Venetian blinds theorem]\label{itervb}
Let $\phi : D \to \R$ be admissible, and let $\Asmall, \Acover$ be disjoint compact subsets of $\R$ such that $\Acover$ is an interval. Let $G$ be a compact curve contained in $\bigcap_{\alpha \in \Acover} D_{\alpha}$. Then for any $\epsilon > 0$, there exists a set $G'$, which is a finite union of disks, such that:
\begin{enumerate}[label=(\alph*)]
    \item $G'\subset \nbhd{G}{\e}$.
    \item $\phi_{\al}(G) \subset \phi_{\al}(G')$ for all $\al \in \Acover$.
    \item $|\phi_{\al}(G')| < \epsilon$ for all $\al \in \Asmall$. 
\end{enumerate} 

\end{thm}

\subsection{Reduction to local version}

We prove Theorem \ref{itervb} by way of the following localized version.

\begin{thm}[Iterated Venetian blinds theorem, local version]\label{itervbloc}
Let $\phi : D \to \R$ be admissible, and let $\Asmall, \Acover$ be disjoint compact subsets of $\R$ such that $\Acover$ is an interval. Furthermore, let $K \subset \R^2$ be a closed ball such that:
\begin{itemize}
    \item $2K \subset \bigcap_{\alpha \in \Acover} D_{\alpha}$.
    \item There exists $\sigma > 0$ such that
    \begin{align}
    \label{eq:theta set 2K}
    |\theta_{\alpsmall}(x_1) - \theta_{\alpcover}(x_2)| > \sigma
    \qquad\text{for all $\alpsmall \in \Asmall$, $\alpcover \in \Acover$, $x_1 \in 2 K \cap D_{\alpsmall}$, $x_2 \in 2 K$}.
    \end{align}
\end{itemize}
Let $G$ be a compact curve contained in $K$. Then for any $\epsilon > 0$, there exists a set $G'$, which is a finite union of line segments, such that:
\begin{enumerate}[label=(\alph*)]
    \item $G' \subset \nbhd{G}{\epsilon}$
    \item $\phi_{\al}(G) \subset \phi_{\al}(G')$ for all $\al \in \Acover$,
    \item $|\phi_{\al}(G')| < \epsilon$ for all $\al \in \Asmall$. 
\end{enumerate} 
\end{thm}
\begin{rmk}
We note that \Cref{itervb} was stated with the output set $G'$ a union of disks, whereas in \Cref{itervbloc} the set $G'$ is a union of line segments.  This is merely a cosmetic difference; it is easy to see that a finite union of segments can be converted into a finite union of disks with the same projection properties.  The reason for the discrepancy is that line segments are what our construction yields, but when we apply \Cref{itervb} in \Cref{Topological base section} to prove Theorem \ref{thm:main}, we use the disk version.
\end{rmk}

\begin{proof}[Proof of \Cref{itervb} assuming \Cref{itervbloc}]
Note that $\bigcap_{\alpha \in \Acover} D_{\alpha}$ is open (since $D$ is open and $\Acover$ is compact) and $G$ is compact. Applying \Cref{thetaseplem}, we can cover $G$ by a finite number of closed balls $\{K_j\}_{j=1}^N$, such that each $K_j$ satisfies the hypotheses of \Cref{itervbloc}. Let $G_j := G \cap K_j$. 

Let $\epsilon > 0$. Apply \Cref{itervbloc} to each piece $G_j$ with $\epsilon/N$ in place of $\epsilon$.  Let $G' = \bigcup_{j=1}^N G_j'.$ Because each set $G_j'$ is a finite union of line segments, it follows that $G'$ is also a finite union of line segments. Since $G_j' \subset \nbhd{G_j}{\epsilon}$ for all $j$, we have $G' \subset \nbhd{G}{\epsilon}$. For each $\al \in \Acover$ and each $j$, we have that 
\[
\phi_{\al}(G) 
=
\bigcup_{j=1}^N
\phi_{\al}(G_j)
\subset 
\bigcup_{j=1}^N
\phi_{\al}(G_j')
=
\phi_{\al}(G'). 
\]
Moreover, if $\al \in \Asmall$, we have that $|\phi_{\al}(G_j')| < \epsilon/N$ for each $j$, so $|\phi_{\al}(G')| < \epsilon$. Thus, $G'$ is a finite union of line segments which satisfies properties (a), (b), and (c) of \Cref{itervb}.  Finally, by replacing each segment with a covering collection of disks of sufficiently small radius, we get a union of disks which satisfies all three desired properties.
\end{proof}

\begin{figure}
\centering
\begin{minipage}[b]{0.45\linewidth}
\includegraphics[page=1]{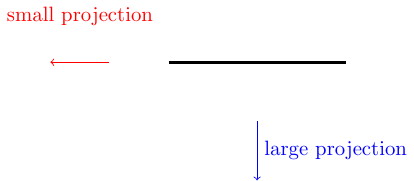}
\caption{Projections of the initial segment}
\label{fig: initial segment}
\end{minipage}
\qquad
\begin{minipage}[b]{0.45\linewidth}
\includegraphics[page=2]{projections.pdf}
\caption{Projections of the initial segment}
\label{fig: blinds}
\end{minipage}
\end{figure}

The remainder of the section is dedicated to proving \Cref{itervbloc}.  This is accomplished using the \textit{Venetian blinds construction}.  Versions of this technique have been used by Falconer \cite{Falconer86} in the linear setting, and by 
authors Chang, McDonald, and Taylor \cite{CMT23} for a particular class of non-linear projection maps.  The basic idea is that if we start with a single line segment, the possible (linear) projections range from a single point to a segment as large as the original one (Figure \ref{fig: initial segment}).  If we replace that segment by a large number of short segments as in Figure \ref{fig: blinds} (called ``blinds''), we obtain a set which still has full projection in the direction orthogonal to the original line, and still has small projection in the direction of the original line, but now has small projection in a new direction.  By iterating this procedure, we can achieve a remarkable degree of control over where projections are large and where they are small.

Our version of the Venetian blinds construction consists of two kinds of steps: an \textit{initial step} in which the curve $G$ is replaced by a union $\mathcal{L}_0$ of line segments pointing in a fixed direction $\theta_0$, and \textit{iterative steps} in which a collection $\mathcal{L}$ of line segments pointing in some direction $\theta$ are replaced by a collection of shorter line segments $\mathcal{L}'$ pointing in direction $\theta + \delta$ (where $\delta$ is a small fixed constant depending on the other parameters). The steps will be done so that, after the initial step, the total length of $\mathcal{L}_0$ will be bounded above by a constant\footnote{Constants may depend on $\sigma$ and $\phi$, but not $\delta$} times the length of $G$, and during each of the iterative steps, the total length $|\mathcal{L}'|$ of $\mathcal{L}'$ is bounded above by $(1 + O(\delta)) |\mathcal{L}|$. Since we will perform $O(1/\delta)$ many iterative steps, we will have $|G'| \leq C_{\sigma}(G)$ since $\lim_{\delta \to 0} C(1 + C \delta)^{1/\delta} < \infty$ for any $C > 0$.

\subsection{The initial and iterative steps}

We will define two sets of angles $\Tcover$ and $\Tsmall$ as follows.  First, the angle set $\Tcover$ is defined by
\begin{align}
\label{eq:Tcover}
\Tcover = \{\theta_\al(x) : x \in 2K, \al \in \Acover\}.
\end{align}
Since $\theta_{\alpha}(x)$ is a continuous function in $(\alpha,x)$ on the connected compact set $\Acover \times 2K$, it follows that $\Tcover$ is a compact interval. We define the angle set $\Tsmall$ to be the complement of the $\sigma$-neighborhood of $\Tcover$ in $\R/\pi\Z$. By construction and \eqref{eq:theta set 2K}, we have
\begin{align}
\label{eq:Tsmall}
\{\theta_\al(x) : x \in 2K \cap D_\alpha, \al \in \Asmall\} &\subset \Tsmall
\\
\label{eq:dist Tcover Tsmall}
\dist(\Tcover, \Tsmall) = \sigma
\end{align}

We write $\Tsmall =[\thetamin, \thetamax]$, and write $\thetatot = \thetamax - \thetamin$, the length of the arc $\Tsmall$.

Working with the parameter sets $\Asmall,\Acover$ directly turns out to be somewhat difficult if the domain $D$ has a complicated shape, so \eqref{eq:Tcover} and \eqref{eq:Tsmall} allow us to avoid this difficulty by working directly with sets of angles $\Tsmall,\Tcover$.

\begin{lem}[Initial step]\label{initialstep}
Let $K$ and $\sigma$ be as in \Cref{itervbloc}, and let $G\subset K$ be a compact curve.  If $N \in \mathbb{N}$ is sufficiently large, we can construct a set $\mathcal{L}_0$ of line segments such that 
\begin{enumerate}[label=(\roman*)]
\item $\mathcal{L}_0$ consists of $N$ line segments of length $\frac{2|G|}{N \sigma}$, each lying in the $\frac{2|G|}{N\sigma}$ neighborhood of $G$;
\item $\phi_{\al}(G) \subset \phi_{\al}(\mathcal{L}_0)$ for all $\al$ in $\Acover$;
\item Each line segment in $\mathcal{L}_0$ points in the direction $\thetamin$.
\end{enumerate}
\end{lem}
\begin{proof}
Write $K = \overline B(x_0,r)$. Partition the curve $G$ into $N$ arcs $\{G_k\}_{k=1}^N$ of length $|G|/N$. We let $x_k$ denote the center of the arc $G_k$. Observe that each arc $G_k$ is contained in the ball $K$. For each $k$, let $L_k$ be the line segment of length $\frac{2|G|}{N \sigma} =: 2\ell_L$ centered at $x_k$ pointing in the direction $\thetamin$. The set $\mathcal{L}_0$ will be the union of the segments $L_k$. Clearly (i) and (iii) are satisfied, and $\mathcal{L}_0$ certainly lies in a $\frac{2|G|}{\sigma N}$ neighborhood of $G$.

We claim that if $N$ is sufficiently large, then (ii) is satisfied.  To show (ii), it suffices to show $\phi_\al(G_k) \subset \phi_\al(L_k)$ for every $k$ and every $\alpha \in \Acover$. Fix $k$ and $\alpha \in \Acover$, and let $x \in G_k$. Since $x \in G_k$, it follows that $\norm{x - x_k} \leq \frac{1}{2} |G_k| = \frac{|G|}{2N}.$ Also $x \in G_k \subset G \subset K$, so for any $y \in B(x,r)$, we have $y \in 2 K$, so $\theta_\alpha(y) \in \Tcover$ (recall  \eqref{eq:Tcover}), so $|\theta_\alpha(y) - \thetamin| \geq \sigma$ (recall \eqref{eq:dist Tcover Tsmall}). If $N$ is chosen sufficiently large so that $4|G|/(N\sigma ) < r$, then
\[
\min\{\ell_L, r\} \geq
\frac{2}{\sigma} \|x_k-x\|
\]
so we can apply \Cref{geomlem phi} (recall \eqref{eq:geomlem 1}) to conclude that $\phi_\al(x) \in \phi_\al(L_k)$.
\end{proof}

\begin{lem}[Iterative step]\label{iterativestep}
Let $L \subset K$ be a line segment of length $l$ pointing in the direction $\theta$. Suppose $\theta+\delta \in \Tsmall$.  If the integer $N \in \mathbb{N}$ is sufficiently large, then there exists a set $\mathcal{L}$ of line segments such that
\begin{enumerate}[label=(\roman*)]
\item  $\mathcal{L}$ consists of $N$ line segments of length $\left(1 +  \frac{\delta}{\sigma} \right) \frac{l}{N}$, each lying in the $(1 + \frac{\delta}{\sigma})\frac{l}{N}$-neighborhood of $L$,
\item $\phi_{\al}(L) \subset \phi_{\al}(\cup \mathcal{L})$ for all $\al$ in $\Acover$,
\item Each line segment in $\mathcal{L}$ points in the direction $\theta + \delta$.
\end{enumerate}
\end{lem}
\begin{proof}
Write $K = \overline B(x_0,r)$. Suppose $L$ is a line segment of length $l$ pointing in the direction $\theta$. Let $\{L_k\}_{k=1}^N$ be the line segments formed by partitioning $L$ into $N$ pieces of equal length, and let $x_k$ denote the center of the line segment $L_k$. For each $k$, let $L_k'$ denote the line segment of length $\left(1 +  \frac{\delta}{\sigma} \right) \frac{l}{N} := 2\ell_L$ centered at $x_k$ pointing in the direction $\theta + \delta$. Let $\mathcal{L}$ denote the union of the line segments $L_k'$. Then $\mathcal{L}$ obviously satisfies conditions (i) and (iii) in the lemma.

We claim that if $N$ is sufficiently large, then (ii) is satisfied.  To show (ii), it suffices to show $\phi_\al(L_k) \subset \phi_\al(L_k')$ for every $k$ and every $\alpha \in \Acover$. Fix $k$ and $\alpha \in \Acover$, and let $x \in L_k$. The choice of $L_k'$ guarantees that the angle between $L_k$ and $L_k'$ is exactly $\delta$. Since $x \in L_k$, it follows that $\norm{x - x_k} \leq \frac{l}{2N}.$ Also $x \in L_k \subset L \subset K$, so for any $y \in B(x,r)$, we have $y \in 2 K$, so $\theta_\alpha(y) \in \Tcover$, so $|\theta_\alpha(y) - (\theta+\delta)| \geq \sigma$ (recall \eqref{eq:dist Tcover Tsmall}). If $N$ is chosen sufficiently large so that $\delta l/(\sigma N) < r$, then
\begin{equation}
    \ell_L \geq \left(1 +  \frac{\delta}{\sigma} \right)  \|x_k-x\| \qquad\text{and}\qquad r \geq \frac{2\delta}{\sigma} \|x_k-x\|.
\end{equation}
so we can apply \Cref{geomlem phi} (recall \eqref{eq:geomlem 2}) to conclude that $\phi_\al(x) \in \phi_\al(L_k')$.
\end{proof}

\subsection{Proof of local iterated Venetian blind theorem}

\begin{proof}[Proof of \Cref{itervbloc}]
Let $\phi : D \to \R$, $\Asmall$, $\Acover$,  $\sigma > 0$, $K \subset \R^2$, $\e > 0$, and $G \subset K$ be as in the statement of Theorem \ref{itervbloc}. Let $\Tcover$ and $\Tsmall = [\thetamin,\thetamax]$ be as in \eqref{eq:Tcover} and \eqref{eq:Tsmall}. 

Since $\phi$ is admissible, it satisfies the rectangle squishing condition (recall \Cref{dfn:rectangle squishing condition}). By applying this condition to $\epsilon$ and $(\Asmall \times 2K) \cap D$, there exists $\delta = \delta(\Asmall, K, \epsilon) > 0$ such that the following holds: Let $L \subset \mathbb{R}^2$ be a line segment of length $|L| \le \delta$. Let $\theta_L$ denote the direction of $L$. Then for all $\alpha \in \Asmall$, we have the implication
\begin{multline}
\label{eq:squishing implication}
(\exists x \in D_\alpha \cap \nbhd{L}{\delta|L|} \cap 2K \text{ such that } |\theta_\alpha(x) - \theta_L| \le \delta)
\\
\implies
|\phi_\alpha(\nbhd{L}{\delta|L|} \cap D_\alpha)| \leq \epsilon |L|
\end{multline}
By making $\delta$ smaller if necessary, we can assume that $\delta < \min(1,\epsilon)$ and $\thetatot := \thetamax-\thetamin$ is a multiple of $\delta$. Let $M = \thetatot/\delta$.

We will inductively construct a sequence of sets $\{\mathcal{L}_i\}_{i=0}^{M}$ and a sequence of natural numbers $\{N_i\}_{i=0}^{M}$ with the following properties:
\begin{enumerate}[label=(\Alph*)]
\item \label{item:number}
The set $\mathcal{L}_i$ is the union of $\prod_{k=0}^i N_k$ line segments, each pointing in the direction $\theta_i := \thetamin + \delta i$ and of length 
\[
\ell_i
=
\frac{1}{\prod_{k=0}^i N_k} \frac{2|G|}{\sigma} \left(1+\frac{\delta}{\sigma}\right)^i 
.
\]
\item \label{item:l_i}
\begin{enumerate}[label=(\roman*)]
    \item For $i = 0$: $\ell_0 < \frac{\delta}{2}$
    \item For $i \geq 1$: $\ell_i < \frac{\delta}{2}\ell_{i-1}$
\end{enumerate}    
\item \label{item:nbhd}
\begin{enumerate}[label=(\roman*)]
    \item For $i = 0$: $\cup \mathcal{L}_0$ lies within the $\ell_0$-neighborhood of $G$.
    \item For $i \geq 1$: $\cup \mathcal{L}_i$ lies within the $\ell_i$-neighborhood of $\cup \mathcal{L}_{i-1}$
\end{enumerate}
\item \label{item:cover}
\begin{enumerate}[label=(\roman*)]
    \item For $i = 0$: If $\al \in \Acover$, then $\phi_{\al}(G) \subset \phi_{\al}(\cup \mathcal{L}_0)$
    \item For $i \geq 1$: If $\al \in \Acover$, then $\phi_{\al}(\cup \mathcal{L}_{i-1}) \subset \phi_{\al}(\cup \mathcal{L}_i)$
\end{enumerate}
\end{enumerate}

For the base case, choose $N_0$ so large that \Cref{initialstep} can be applied to $G$. Apply \Cref{initialstep} to arrive at a set $\mathcal{L}_0$ consisting of $N_0$ many line segments pointing in the direction $\thetamin$, each of which has length 
$
\ell_0 := \frac{2|G|}{N_0 \sigma}
$. By construction and by \Cref{initialstep}, properties \ref{item:number}, \ref{item:nbhd}, \ref{item:cover} are satisfied. By choosing $N_0$ larger if needed, \ref{item:l_i} is satisfied.

Next, for the inductive step, suppose we have $\mathcal{L}_{i-1}$ satisfying the properties above. We choose $N_i$ sufficiently large that we can apply \Cref{iterativestep} (with $\theta=\theta_{i-1}$ so that $\theta+\delta=\theta_i$) to each line segment in $\mathcal{L}_{i-1}$. We let $\mathcal{L}_i$ be the set of line segments obtained this way.  By \Cref{iterativestep}, every line segment in $\cL_i$ has length 
$
\ell_i
=
\left(1+\frac{\delta}{\sigma}\right) \frac{\ell_{i-1}}{N_i}
$,
which implies \ref{item:number}. We choose $N_i$ sufficiently large so that \ref{item:l_i} is satisfied. By \Cref{iterativestep}, \ref{item:nbhd} and \ref{item:cover} are satisfied. This completes the inductive construction.

The set $G' := \cup \mathcal{L}_{M}$ is a union of line segments pointing in the direction $\thetamin + M \delta = \thetamax$. By \ref{item:cover}, it follows that $\phi_{\al}(G) \subset \phi_{\al}(G')$ for all $\al \in \Acover$. By \ref{item:nbhd}, $G'$ is contained in the ($\sum_{i=0}^M \ell_i$)-neighborhood of $G$. Note that by \ref{item:l_i}, $\sum_{i=0}^M \ell_i \leq \sum_{i=0}^M (\frac{\delta}{2})^i \ell_0 < \delta \leq \epsilon$, so $G' \subset \nbhd{G}{\epsilon}$.

It remains to check that $\phi_{\al}(G')$ is small for $\al \in \Asmall$. We start by making a few observations.

The line segments in $\{\mathcal{L}_i\}_{i=0}^M$ form a tree structure. For $L \in \mathcal{L}_i$, let $\mathcal{L}_{i+1}(L) \subset \cL_{i+1}$ denote its children obtained via \Cref{iterativestep}. More generally, if $L \in \mathcal{L}_i$ and $j > i$, we write $\mathcal{L}_j(L)$ for the union of the descendants of $L$ in $\mathcal{L}_j$.

We claim that for every $L \in \bigcup_{i=0}^M \mathcal{L}_i$,
\begin{gather}
\label{eq:L_M in nbhd}
    \mathcal{L}_M(L) \subset \nbhd{L}{\delta |L|} 
    \\
\label{eq:L leq sum leaves}
    |L| \leq \sum_{L' \in \cL_M(L)} |L'|
\end{gather}
To see \eqref{eq:L_M in nbhd}, observe that if $L \in \mathcal{L}_i$, then by property (C), the set $\mathcal{L}_M(L)$ lies in the ($\sum_{j=i+1}^M \ell_j$)-neighborhood of $L$. Note that
$
\sum_{j=i+1}^{M} \ell_j
\leq
\sum_{j=i+1}^{M} \left(\frac{\delta}{2}\right)^{j-i} \ell_i
\leq  \delta \ell_i
,
$ so we have \eqref{eq:L_M in nbhd}. For \eqref{eq:L leq sum leaves}, observe that for every $L \in \bigcup_{i=0}^M \mathcal{L}_i$, we have $|L| \leq \sum_{L' \text{ child of } L} |L'|$, so by induction, \eqref{eq:L leq sum leaves} holds.

Now we check that $\phi_{\al}(G')$ is small for $\al \in \Asmall$. Fix $\al \in \Asmall$. 

For $L \in \mathcal{L}_i$, we say that $L$ is \emph{good} if $|\phi_\alpha(\cL_M(L))| \leq \epsilon |L|$. Note that if $L$ is not good, then by \eqref{eq:L_M in nbhd} and the rectangle squishing condition (specifically, the contrapositive of \eqref{eq:squishing implication}), there does not exist $y \in \nbhd{L}{\delta |L|} \cap D_\al$ such that $|\theta_\al(y) - \theta_i| \leq \delta$.

We claim that every $L \in \cL_M$ has a good ancestor or is itself good. Suppose for contradiction that this is false for some $L \in \cL_M$. Take some $y \in L \cap D_\alpha$. (If $L \cap D_\alpha = \emptyset$, then $L$ is good.) For all $i \in \{0, \ldots, M\}$, we have $|\theta_\alpha(y) - \theta_i| > \delta$. However, $\theta_\alpha(y) \in \Tsmall$, and $\{\theta_i\}_{i=0}^M$ forms a $\delta$-net for $\Tsmall$, yielding a contradiction.

Let $\mathcal{G} \subset \bigcup_{i=0}^M \mathcal{L}_i$ be the set of good line segments that have no strictly earlier good ancestor. By the claim in the previous paragraph, $\{\mathcal{L}_M(L)\}_{L \in \mathcal{G}}$ forms a partition of $\cL_M$, so $G' = \bigsqcup_{L \in \mathcal G} \cL_M(L)$. By \eqref{eq:L leq sum leaves},
\[
|\phi_\al(G')|
\leq
\sum_{L \in \mathcal G}
|\phi_\alpha(\cL_M(L))|
\leq
\epsilon 
\sum_{L \in \mathcal G}
|L|
\leq
\epsilon 
\sum_{L \in \mathcal G}
\sum_{L' \in \cL_M(L)}
|L'|
=
\epsilon|G'|
.
\]

Finally, note that by \ref{item:number}
\[
|G'|
=
\ell_M \prod_{k=0}^M N_k
=
\frac{2|G|}{\sigma} \left(1+\frac{\delta}{\sigma}\right)^{\thetatot/\delta}
\leq \frac{2|G|}{\sigma} e^{\thetatot/\sigma}
=:
C
\]
This shows that $|\phi_\alpha(G')| \leq  C \epsilon$ for all $\alpha \in \Asmall$. Crucially, $C$ is independent of $\epsilon$. This completes the proof.
\end{proof}

\section{The topological base argument}
\label{Topological base section}

\subsection{Setup} 
In this section, we complete the proof of our main theorem, \Cref{thm:main}. As noted in \Cref{rmk:rectangle squishing condition}, we will prove the following slightly more general version:

\begin{thm}[Generalized prescribed projection theorem]
\label{thm:main v2}
Let $\phi : D \to \R$ be admissible. Suppose $\{F_\al\}_{\al\in\R}$ is a family of subsets of $\R$ such that $F_\al\subset \phi_\al(D_\al)$ for each $\al$, and
$
\bigcup_{\alpha\in\R} (\{\alpha\}\times F_\alpha)
$
is a measurable subset of $\R^2$.  Then there exists a Borel set $E\subset \R^2$ such that
\begin{align*}
&\phi_\alpha(E \cap D_\alpha) \supset F_\alpha &&\text{for all } \alpha \in \R
\\
&|\phi_\alpha(E \cap D_\alpha)\setminus F_\alpha|=0
&&\text{for almost every } \alpha \in \R.
\end{align*}
\end{thm}

The rough idea of the proof is as follows. Theorem \ref{itervb} can be seen as a primitive version of Theorem \ref{thm:main v2}, where our prescription sets $F_\al$ are required to be of the form $\phi_\al(G)$ for some curve $G$ and our errors are allowed to be of measure $\e$ rather than measure zero.  While restricting to sets $F_\al$ of this form may seem like a significant loss of generality, the key idea is that these sets generate a topological basis for the open sets in $\R^2$.  Thus, Theorem \ref{itervb} is a stepping stone to being able to handle open sets.  From there, it is not difficult to extend to general measurable sets.

Before we proceed, we establish some new notation that will be helpful.  Let $\phi : D \to \R$ be admissible. For $E \subset \R^2$, define
\[
\phi^*(E)=\bigcup_{\alpha\in \R} \{\alpha\}\times \phi_\alpha(E).
\]
Note that for some values of $\alpha$ the set $E$ may not be contained in the domain $D_\alpha$, in which case we make the convention $\phi_\alpha(E)=\phi_\alpha(E\cap D_\alpha)$. It follows immediately from the definition that
\begin{equation}
\label{stardef}
(\alpha,\beta)\in \phi^*(E)\hspace{.25in} \text{if and only if} \hspace{.25in} \beta\in \phi_\alpha(E).
\end{equation}
To summarize, each of the projections $\phi_\al$ maps a subset of $\R^2$ to a subset of $\R$, whereas $\phi^*$ is a set function which takes subsets of $\R^2$ to subsets of $\R^2$.  Going forward, it will be helpful to think of the input and output space of $\phi^*$ as two different copies of $\R^2$.  Thus, when we write things like $\phi^*(E)\supset F$, we will write elements of $E$ as $x=(x_1,x_2)$ and elements of $F$ as $(\al,\be)$, and we will think of ``$x$-space'' and ``$(\al,\be)$-space'' as different copies of $\R^2$.

\begin{figure}
\centering
\begin{minipage}[b]{0.45\linewidth}
\centering
\includegraphics[page=1]{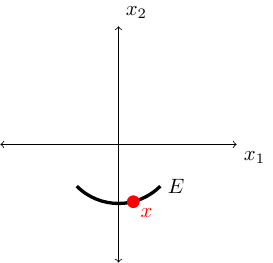}
\end{minipage}
\begin{minipage}[b]{0.45\linewidth}
\centering
\includegraphics[page=2]{phi_star.pdf}
\end{minipage}
\caption{A point $x$ and set $E$ in $x$-space (left), and the sets $\phi^*(x)$ and $\phi^*(E)$ in $(\al,\be)$-space (right)}
\label{fig: phistar}
\end{figure}

As an example, consider the projections $\phi_\al(x)=x_2+\sqrt{1-(\al-x_1)^2}$, defined so that we have $\be=\phi_\al(x)$ if and only if $x\in D_\al$ and $|x-(\al,\be)|=1$ and $\beta > x_2$.  For concreteness, take $D_\al=[\al-1/\sqrt{2},\al+1/\sqrt{2}]$, so that $\be=\phi_\al(x)$ holds if and only if $(\al,\be)$ is in the upper quarter circle centered at $x$.  In \Cref{fig: phistar}, the left image shows a point (in red) and a curve in the input space.  The right image shows the curve $\phi^*(x)$ and the set $\phi^*(E)$, which is a union of curves of the form $\phi^*(x)$ as $x$ ranges over $E$.

\begin{thm}
\label{prescribedprojection}
Let $\phi$ be admissible, and let $\phi^*$ be the set function defined above.  For any measurable set $F\subset \phi^*(\R^2)$, there exists a set $E\subset \R^2$ such that
\begin{enumerate}[label=(\alph*)]
\item $\phi^*(E)\supset F$,
\item $|\phi^*(E)\setminus F|=0$.
\end{enumerate}
\end{thm}

\begin{proof}[Proof of \Cref{thm:main v2} assuming \Cref{prescribedprojection}]
Let $\{F_\al\}_\al$ be as in the statement of \Cref{thm:main v2}.  By assumption, the set
\[
F:=\bigcup_{\al\in \R} \{\al\}\times F_\al
\]
is measurable.  Let $E$ be as given by Theorem \ref{prescribedprojection}.  For any $\al$, if $\be\in F_\al$ then $(\al,\be)\in F$, hence $(\al,\be)\in \phi^*(E)$ by (a).  By (\ref{stardef}) it follows that $\be\in \phi_\al(E)$, so $\phi_\al(E)\supset F_\al$.  Finally, by Fubini we have
\[
0=|\phi^*(E)\setminus F|=\int_{\R} |\phi_\al(E)\setminus F_\al|\,d\al,
\]
so $|\phi_\al(E)\setminus F_\al|=0$ for almost every $\al$.
\end{proof}
The goal of the remainder of this section is to prove \Cref{prescribedprojection}.

\subsection{Properties of $\phi^*$}
The purpose of this subsection is to prove some basic properties of the set function $\phi^*$ which will be needed in our construction.
\begin{lem}
\label{lem: openness of phi star}
Let $\phi$ be admissible, and define
\begin{equation}
\label{eq: define auxiliary map}
\begin{split}
\Phi:D&\to \R^2, \\
\Phi(\al,x)&=(\al,\phi_\al(x)).
\end{split}
\end{equation}
Then $\Phi$ is an open mapping.  In particular, if $U$ is open, then $\phi^*(U)$ is open.
\end{lem}
\begin{proof}
The derivative of $\Phi$ is the $2 \times 3$ matrix
\begin{align}
\label{eq:Dphi matrix v2}
D\Phi(\alpha,x)
=
\begin{pmatrix}
    1 & 0 & 0
    \\
    \partial_\alpha\phi_\alpha(x) & \partial_{x_1} \phi_\alpha(x) & \partial_{x_2} \phi_\alpha(x)
\end{pmatrix}
\end{align}
By assumption $\nabla \phi_\alpha(x)\neq 0$ for all $\alpha$ and all $x$, so $D\Phi$ has full rank $2$ everywhere. By the constant rank theorem, $\Phi$ is an open mapping.  The ``in particular'' statement follows from the observation that $\phi^*(U)=\Phi((\R\times U) \cap D)$ and that $(\R\times U) \cap D$ is open. 
\end{proof}

\begin{lem}
\label{difference}
Let $\phi$ be admissible.  If $E\subset \R^2$ is a connected Borel set and $\overline{E}$ is its closure, then $|\phi^*(\overline{E})\setminus \phi^*(E)|=0$.
\end{lem}
\begin{proof}
We first observe that for each $\alpha$, the slice $\phi_\alpha(\overline{E})\setminus \phi_\alpha(E)$ contains at most two points. Indeed, since $\phi_\alpha$ is continuous, we have
\[
\phi_\alpha(\overline{E})\setminus \phi_\alpha(E) \subset \overline{\phi_\alpha(E)}\setminus \phi_\alpha(E).
\]
Since $E$ is connected and $\phi_\alpha$ is continuous, $\phi_\alpha(E)$ is an interval.  This proves the claim.  
By Fubini, the conclusion of the lemma will follow if we prove $\phi^*(E)$ is measurable.  Consider the map $\Phi$ defined in (\ref{eq: define auxiliary map}).  For any set $E$, we have $\phi^*(E)=\Phi((\R\times E)\cap D)$.  In particular, if $E$ is Borel, then $\phi^*(E)$ is the image of the Borel set $(\R\times E)\cap D$ under the continuous function $\Phi$.  Since the continuous image of a Borel set is always measurable, this completes the proof.
\end{proof}

\begin{dfn}
Let $\Ga\subset \R^2$ be a curve.  A \textbf{sub-curve} of $\Ga$ is a compact and connected subset $\Ga'\subset\Ga$ which contains more than one point.
\end{dfn}

For the next lemma, it may be helpful to note that if $G$ is a sub-curve of $\phi_{\alpha_0}^{-1}(\beta_0)$, then $G \subset D_{\alpha_0}$. Also, for any $I \subset \R$ and $E \subset \R^2$
\[
(I \times \R) \cap \phi^*(E) = \bigcup_{\alpha \in I} \{\alpha\} \times \phi_\alpha(E).
\]

\begin{lem}\label{differenceset}
Let $\phi$ be admissible, let $(\alpha_0, \beta_0) \in \mathbb{R}^2$, and let $G$ be a sub-curve of $\phi_{\alpha_0}^{-1}(\beta_0)$. Then we have the following.
\begin{enumerate}[label=(\alph*)]
\item $\phi^*(G)$ has nonempty interior. More precisely, letting $G^\circ$ denote the relative interior of the curve $G$, for any $x \in G^{\circ}$ we have
\[\phi^*(x) \setminus \{(\alpha_0, \beta_0)\} \subset \inter \phi^*(G).\]
In particular, a point $(\alpha, \beta) \in \phi^*(G)$ is an interior point whenever $\beta$ is an interior point of the interval $\phi_{\alpha}(G)$. 
\item If $\eta > 0$ is such that $(\alpha_0-\eta,\alpha_0+\eta) \times G \subset D$, then the boundary of $((\alpha_0 - \eta, \alpha_0 + \eta) \times \mathbb{R}) \cap \phi^*(G)$ has measure zero.
\item Let $K \subset D$ be a compact set. If $\eta > 0$ is such that $(\alpha_0-\eta,\alpha_0+\eta) \times G \subset K$, then 
\[
((\alpha_0-\eta,\alpha_0+\eta) \times \mathbb{R}) \cap \phi^{*}(G)
\subset
\R \times (\beta_0-C_K\eta,\beta_0+C_K\eta),
\]
where
$
C_K
=
\sup\left\{|\partial_\alpha \phi(\alpha,x)| : (\alpha,x) \in K
\right\}
.
$
\end{enumerate}
\end{lem}

\begin{proof}
\textbf{Proof of (a).} Define $\Phi$ as in (\ref{eq: define auxiliary map}), and recall the derivative is the $2 \times 3$ matrix given in (\ref{eq:Dphi matrix v2}).
Let $\Gamma = \phi_{\alpha_0}^{-1}(\beta_0)$ and for $x \in \Gamma$, let $T_x\Gamma$ denote the tangent space to $\Gamma$ at $x$. The derivative of the restriction  $\Phi|_{(\R \times \Gamma) \cap D} : (\R \times \Gamma) \cap D \to \R^2$ is the restriction of \eqref{eq:Dphi matrix v2} to the $2$-dimensional subspace $\R \times T_x \Gamma$. We claim that the angle monotonicity condition implies that
\begin{align}
\label{eq:deriv-full-rank}
\text{the derivative $D(\Phi|_{(\R \times \Gamma)\cap D})(\alpha,x)$ has full rank for all $(\alpha,x) \in ((\R\setminus\{\alpha_0\}) \times \Gamma) \cap D$.}
\end{align}
Indeed, note that $\R \times T_x\Gamma$ is spanned by $(1,0,0)$ and $(0, \nabla\phi_{\alpha_0}(x)^\perp)$, and the matrix \eqref{eq:Dphi matrix v2} maps these two vectors to $(1,\partial_\alpha\phi_\alpha(x))$ and $(0,\nabla\phi_{\alpha_0}(x)^\perp \cdot \nabla\phi_{\alpha}(x))$, respectively. These two vectors are linearly independent, since the angle monotonicity condition implies $\nabla\phi_{\alpha_0}(x)^\perp \cdot \nabla\phi_{\alpha}(x) \neq 0$. Thus, we have \eqref{eq:deriv-full-rank}.

Thus, by the constant rank theorem, for all $(\alpha,x) \in ((\R\setminus\{\alpha_0\}) \times \Gamma) \cap D$, if $\e$ is sufficiently small, then
\[
\Phi(B((\alpha,x),\e) \cap (\R \times \Gamma)) \text{ is open}.
\]
If $(\alpha,x) \in ((\R \setminus \{\alpha_0\}) \times G^\circ) \cap D$, then we can choose $\e$ small enough so that $B((\alpha,x),\e) \cap (\R \times \Gamma) \subset (\R \times G) \cap D$, which shows $\Phi(\alpha,x) \in \inter \Phi((\R \times G)\cap D) = \inter \phi^*(G)$.

\textbf{Proof of (b).} 
Suppose $\eta > 0$ is such that $(\alpha_0-\eta,\alpha_0+\eta) \times G \subset D$. If $\alpha \in (\alpha_0 - \eta, \alpha_0+\eta)$, then $G \subset D_{\alpha}$. Since $G$ is compact and connected, $\phi_{\alpha}(G)$ is a closed (and possibly degenerate) interval, so the boundary of $\phi_\alpha(G)$ consists of at most two points. Therefore, Fubini's theorem implies that the boundary of $((\alpha_0 - \eta, \alpha_0 + \eta) \times \mathbb{R}) \cap \phi^*(G)$ has Lebesgue measure zero. 

\textbf{Proof of (c).} 
Let $K \subset D$ be a compact set and suppose $\eta > 0$ is such that $(\alpha_0-\eta,\alpha_0+\eta) \times G \subset K$. For any $x \in G$ and any $\alpha' \in (\alpha_0-\eta,\alpha_0+\eta)$, 
\[
|\phi_{\alpha'}(x) - \beta_0| = |\phi_{\alpha'}(x) - \phi_{\alpha_0}(x)| \leq C_K|\alpha' - \alpha_0| < C_K\eta.
\]
Thus $\phi_{\alpha'}(G) \subset (\beta_0 - C_K\eta, \beta_0 + C_K\eta)$, which implies (c).
\end{proof}

\subsection{Systems of disks}

The mechanism we use to extend Theorem \ref{itervb} to a proof of Theorem \ref{prescribedprojection} is contained in the following definition.

\begin{dfn}
\label{definition: system-of-disks}
We call $\{\cB_m:m\in\N\}$ a \textbf{system of disks} corresponding to $\newA$ if the following conditions hold:
\begin{enumerate}[label=(\roman*)]
\item $\phi^*(\cup \cB_m)\supset \newA$,
\item $|(\phi^*(\cup \cB_m)\setminus \newA) \cap ([-m, m] \times \R)|\to 0$ as $m\to\infty$,
\item For any $B\in\cB_{m-1}$ and any $y\in \phi^*(B)\cap \newA$, there exists $B'\in\cB_m$ with $B'\subset B$ and $y\in\phi^*(B')$. 
\end{enumerate}
We say that $F$ \textbf{admits a system of disks} if such a collection exists.
\end{dfn}

The motivation for \Cref{definition: system-of-disks} is that it provides a mechanism to build a set $E$ that satisfies the conditions of Theorem \ref{prescribedprojection}, as we show here.

\begin{lem}
\label{disks}
Let $\newA\subset \R^2$ be measurable.  If $F$ admits a system of disks, then there exists a set $E$ such that $\phi^*(E)\supset \newA$ and $|\phi^*(E)\setminus \newA|=0$.
\end{lem}
\begin{proof}
 Let $\{\cB_m:m\in\N\}$ be a system of disks corresponding to $\newA$, and define
 \[
 \widetilde{E}_m:=\bigcup_{B\in\cB_m}\overline{B} \hspace{.25in}\text{and}\hspace{.25in}  E:=\bigcap_m \widetilde{E}_m.
 \]
First we show $\phi^*(E)\supset \newA$.  Let $y=(\alpha,\beta)\in \newA$; by \eqref{stardef}, it suffices to prove $\beta\in \phi_\alpha(E)$.  By condition (i) of \Cref{definition: system-of-disks}, there exists $B_1\in \mathcal{B}_1$ with $y\in \phi^*(B_1)$. By condition (iii) of \Cref{definition: system-of-disks}, there exists $B_2\in \mathcal{B}_2$ with $B_2\subset B_1$ and $y\in \phi^*(B_2)$.    Continuing in this way, we obtain a nested sequence of disks $B_1\supset B_2\supset B_3\supset\cdots$ such that $y\in \phi^*(B_m)$ for each $m$.  This means that for each $m$, there is an $x_m\in B_m$ such that $\phi_\alpha(x_m)=\beta$.  By compactness of the closed disks, the sequence $\{x_m\}$ has a limit point $x\in \bigcap_m \overline{B}_m \subset E$.  By continuity of $\phi_\alpha$, we have $\beta=\phi_\alpha(x)$, hence $y\in \phi^*(\{x\}) \subset \phi^*(E)$.

Now we show $|\phi^*(E)\setminus \newA|=0$. We observe for any fixed $m$ we have $\phi^*(E)\setminus \newA\subset \phi^*(\widetilde{E}_m)\setminus \newA$.  We then use the following decomposition: letting $E_m=\cup \mathcal{B}_m$,
\begin{align*}
&|(\phi^*(E)\setminus \newA) \cap ([-m,m] \times \R)|
\\
&\qquad\leq 
|(\phi^*(\widetilde{E}_m)\setminus \phi^*(E_m))|
+
|(\phi^*(E_m)\setminus \newA) \cap ([-m,m] \times \R)|
.
\end{align*}
The first term on the right is zero by \Cref{difference}, and the second term tends to zero as $m\to\infty$ by (ii) of \Cref{definition: system-of-disks}. This implies $|\phi^*(E)\setminus \newA|=0$.
\end{proof}

\Cref{prescribedprojection} (and hence \Cref{thm:main v2}) can now be reduced to the following lemma.

\begin{lem}
\label{lem:prescribedprojection with additional assumptions}
Every measurable set $\newA \subset \phi^*(\R^2)$ admits a system of disks.
\end{lem}

\subsection{Constructing a topological base}
Our approach to proving \Cref{lem:prescribedprojection with additional assumptions} is modeled after the one introduced by Falconer \cite{Falconer86} and extended by  
Chang, McDonald and Taylor \cite{CMT23}.  We want to show that for any pair $(B,V)$, there exists a topological base for $V$ consisting of sets which satisfy an approximate version of \Cref{prescribedprojection} (approximate in the sense that measure zero error is replaced by measure $\e$ error).  Our main ingredient in constructing such a base is \Cref{itervb}.  

To get a feel for how this will work, consider again the example $\phi_\al(x)=x_2+\sqrt{1-(\al-x_1)^2}$ with $D_\al=[\al-1/\sqrt{2},\al+1/\sqrt{2}]$, depicted in \Cref{fig: phistar}.  The set $E$ depicted on the left in that figure happens to be a single fiber of the projection, namely $\phi_0^{-1}(0)$.  The set $\phi^*(\phi_0^{-1}(0))$ shown in gray on the right of \Cref{fig: phistar} has significant interior.  More precisely, while $(0,0)$ is not itself an interior point, we may intersect with a vertical strip $(0,\eta)\times \R$ for $0<\eta\ll 1$ and remove the upper and lower boundary, we get an open set.  By continuity, we can ensure the diameter of this set is small by taking $\eta$ small enough.  If we do the same thing with a set of the form $\phi^*(\phi_\al^{-1}(\be))$ for arbitrary $(\al,\be)$ in place of $(0,0)$, we get a translate of the previous set.  Since this process yields sets which can be translated and made arbitrarily small, we can build a topological base this way.  Moreover, since our building blocks are the types of sets that can be prescribed as projections by \Cref{itervb}, the base constructed this way will have desirable properties.

\begin{lem}[Topological base lemma]
\label{base}
Let $(B,V)$ be a pair of open sets with $\phi^*(B)\supset V$.  There is a base $\mathcal{U}$ for the topology of $V$ such that for any $U\in\mathcal{U}$, $N > 0$, and $\e>0$, there exists a set $E\subset B$ which satisfies 
\begin{enumerate}[label=(\alph*)]
\item $E$ is a finite union of open disks,
\item \label{item:baselemma cover} $U\subset \phi^*(E)$,
\item \label{item:baselemma small} $|(\phi^*(E)\setminus U) \cap ([-N, N] \times \mathbb{R})| < \e$.
\end{enumerate}
\end{lem}
\begin{proof}
If $y=(\alpha_0,\beta_0)\in \R^2$, $G$ is a sub-curve of $\phi_{\alpha_0}^{-1}(\beta_0)$, and $\eta > 0$, we define
\[
U_{y,G,\eta}=\operatorname{int}\,[((\alpha_0+\eta,\alpha_0+2\eta)\times \R)\cap \phi^*(G)].
\] 
Let $\mathcal{U}$ be the set of all $U_{y,G,\eta}$ with parameters satisfying: 
\begin{enumerate}
    \item \label{item:baselemma sub-curve} $G$ is a sub-curve of $\phi_{\alpha_0}^{-1}(\be_0)$, $G\subset B$,
    \item \label{item:baselemma eta small} $0 < \eta < \frac{1}{4} \eta_0(G)$, where $\eta_0(G) := \sup\{\eta' \geq 0: (\alpha_0-\eta', \alpha_0+\eta') \times G \subset D\}$,
    \item \label{item:baselemma U in V} $U_{y,G,\eta}\subset V$.
\end{enumerate}
We first show that $\mathcal{U}$ is a topological base for $V$.  Let $z=(\al,\be)\in V$. Let $V'$ be an open set such that $z \in V' \subset V$.  We must show that there exist $y,G,\eta$ as above such that $z\in U_{y,G,\eta} \subset V'$.  Since $V \subset \phi^*(B)$, there exists some $x\in B \cap \phi_\alpha^{-1}(\beta)$.  Let $B'\subset B\cap D_\al$ be a compact ball centered at $x$, let $I$ be a compact neighborhood of $\al$ with $B'\subset D_{\al'}$ for every $\al'\in I$, let $K=I\times B'$, and let
\[
C_K
=
\sup\left\{|\partial_\alpha \phi(\alpha',x')| : (\alpha',x') \in K
\right\}
.
\]
We may choose $\eta$ sufficiently small so that
\begin{equation}
\label{eq: eta small 1}
B((\alpha, x), 10 \eta) \subset D \cap (\mathbb{R} \times B),
\end{equation}
and so that
\begin{equation}
\label{eq: eta small 2}
(\alpha-2\eta,\alpha+2\eta) \times (\be-4C_K\eta,\be+4C_K\eta)\subset V',
\end{equation}
Indeed, both (\ref{eq: eta small 1}) and (\ref{eq: eta small 2}) are possible since $B,V',D$ are all open sets.  Having chosen $\eta$, it remains to choose $y$ and $G$, and to prove that the resulting set contains $z$ and is in $\mathcal{U}$.  Let $\alpha_0 = \alpha - \frac{3}{2} \eta$, let $\beta_0 = \phi_{\alpha_0}(x)$, and let $y=(\al_0,\be_0)$.  Finally, let $G$ be a sub-curve of $\phi_{\alpha_0}^{-1}(\beta_0)$ contained in $B(x, \eta)$ with $x \in G^\circ$.  Since $z \in \phi^*(x)$ and $x \in G^\circ$, \Cref{differenceset} implies that $z$ is an interior point of $\phi^*(G)$.  Since $z \in \{\alpha\} \times \R \subset (\alpha_0+\eta,\alpha_0+2\eta) \times \R$, we have $z\in U_{y,G,\eta}$.  To show $U_{y,G,\eta}\in\mathcal{U}$, we first observe that \eqref{item:baselemma sub-curve} holds by construction.  Next, we note that by \eqref{eq: eta small 1}, we have $(\alpha_0-\frac{9}{2}\eta, \alpha_0 + \frac{9}{2}\eta) \times G \subset D$.  Therefore, $\eta_0(G) \geq \frac{9}{2} \eta > 4 \eta$, so $\eta < \frac{1}{4} \eta_0(G)$, proving \eqref{item:baselemma eta small}.  Finally, by definition of $C_K$, we have
\[
|\be-\be_0|=|\phi_\al(x)-\phi_{\al_0}(x)|\leq C_K |\alpha - \alpha_0| = C_K \cdot \frac{3}{2}\eta.
\]
This, together with (\ref{eq: eta small 2}) and \Cref{differenceset}(c) (with $2\eta$ in place of $\eta$) implies
\[
U_{y,G,\eta}\subset (\al_0+\eta,\al_0+2\eta)\times (\be_0-2C_K\eta,\be_0+2C_K\eta)\subset V',
\]
so \eqref{item:baselemma U in V} is satisfied. This completes the proof that $\mathcal{U}$ is a base; it remains to show that $\mathcal{U}$ satisfies properties (a)--(c).

Fix  $U=U_{y,G,\eta}\in\mathcal{U}$, $N > 0$, and $\e >0$. Now we need to construct a set $E \subset B$ satisfying (a), (b), (c). Let $\alpha_0$ be the first coordinate of $y$. Since $4\eta < \eta_0(G)$, it follows that 
\begin{align}
\label{eq:G subset B subset D}
(\alpha_0-4\eta, \alpha_0+4\eta) \times G \subset D.
\end{align}
Let $\delta_0>0$ be such that 
\begin{align}
[\alpha_0 -3\eta, \alpha_0+3\eta] \times \overline{\nbhd{G}{\delta_0}} \subset D
\end{align}
and
\begin{align}
\label{eq:nbhd G in B}
\nbhd{G}{\delta_0} \subset B
\end{align}

Let $\delta \in (0, \min(\eta, \delta_0))$ be a small constant to be chosen later. (At the end of the proof, we will see that we can take $\delta = \frac{1}{2} \min(\eta,\delta_0,\frac{\epsilon}{10 C_0\eta + 2N})$.) Let 
\begin{align*}
    I&=(\alpha_0+\eta,\alpha_0+2\eta)
    \\
    I'&=(\alpha_0+\eta-\de,\alpha_0+2\eta+\de)
\end{align*}
(so that $I'$ is the $\de$-neighborhood of $I$). Because of \eqref{eq:G subset B subset D}, we can apply \Cref{itervb} with $\Acover:=\overline{I}$ and $\Asmall:= [-N,N]\setminus I'$ and $\delta$. This gives us a set $E\subset \nbhd{G}{\de}$ which is a finite union of balls satisfying 
\begin{align}
    |\phi_\alpha(E)|\leq \delta \qquad &\text{for } \alpha\in [-N,N]\setminus I'
    \\
    \phi_\alpha(E)\supset \phi_\alpha(G) \qquad &\text{for } \alpha\in \overline I.
\end{align}

By \eqref{eq:nbhd G in B}, $E \subset \nbhd{G}{\de} \subset \nbhd{G}{\delta_0} \subset B$. By construction, (a) is satisfied. To prove (b), let $z=(\alpha,\beta)\in U = \operatorname{int}((I\times \R)\cap \phi^*(G))$. It follows that $\alpha\in I\subset \Acover$, so $\beta \in\phi_\alpha(G) \subset \phi_\alpha(E)$. This proves $U \subset \phi^*(E)$.

Now we show that if $\delta$ is sufficiently small, then $E$ satisfies \ref{item:baselemma small}. We use $U \subset I \times \R$ and split $[-N,N]$ into three parts to obtain
\begin{equation}\label{eq:Phi E minus U}
\begin{split}
|(\phi^*(E)\setminus U) \cap ([-N, N] \times \mathbb{R})|
&= 
 |(\phi^*(E)\setminus U) \cap (I \times \R)| \\
&\qquad +
 |\phi^*(E) \cap ((I'\setminus I) \times \R)| \\
&\qquad +
|\phi^*(E) \cap (([-N,N]\setminus I')\times \R)|
\end{split}\end{equation}
We now bound each of these three terms separately. 

We first make a claim that will be useful for the first two terms of \eqref{eq:Phi E minus U}. Define
\[
C_0 = \sup \{ \|\nabla \phi(\alpha,x)\| : (\alpha,x) \in [\alpha_0-3\eta,\alpha_0+3\eta] \times \overline{\nbhd{G}{\delta_0}} \}.
\]
We claim that 
\begin{equation}
\label{eq:phi nbhd G minus phi G}
|\phi_\alpha(\nbhd{G}{\delta}) \setminus \phi_\alpha(G)| \leq 2C_0\delta \qquad\text{for all } \alpha \in [\alpha_0-3\eta, \alpha_0+3\eta].
\end{equation}
Indeed, suppose $x\in \nbhd{G}{\de}$, and let $x'\in G$ be such that $|x-x'|< \de$. Since $\delta < \delta_0$, the mean value theorem implies
$
|\phi_\alpha(x)-\phi_{\alpha}(x')| \leq C_0|x-x'| < C_0\de
$,
so $\phi_\alpha(x) \in \nbhd{\phi_\alpha(G)}{C_0\de}$. This proves that
\[
\phi_\al(\nbhd{G}{\de})\subset \nbhd{\phi_\al(G)}{C_0\de}.
\]
Next, since $G$ is connected and compact and $\phi_\al$ is continuous, $\phi_\al(G)$ is an interval of finite length, and the $C_0\de$-neighborhood of $\phi_\al(G)$ is an interval which has been extended by $C_0\de$ on each end. This proves the claim \eqref{eq:phi nbhd G minus phi G}.

For the first term in \eqref{eq:Phi E minus U}, we use the fact that $E\subset \nbhd{G}{\delta}$, the fact that $U \subset (I \times \R) \cap \phi^*(G) \subset \overline U$, and the fact that $|U| = |\overline U|$ (which follows from \Cref{differenceset}(b) and the fact that $U \subset (\alpha_0 + \eta, \alpha_0 + 2 \eta) \times \mathbb{R}$ and \eqref{eq:G subset B subset D})
to obtain
\begin{equation}
\label{first term first bound}
|(\phi^*(E)\setminus U) \cap (I \times \R)|\leq|(\phi^*(\nbhd{G}{\de})\setminus \phi^*(G)) \cap (I \times \R)|.
\end{equation}

Therefore, \eqref{eq:phi nbhd G minus phi G} and \eqref{first term first bound} together imply
\begin{equation}
\label{first term third bound}
|(\phi^*(E)\setminus U) \cap (I \times \R)|
\leq \int_{I}|\phi_\alpha(\nbhd{G}{\delta}) \setminus \phi_\alpha(G)| \, d\al
\leq 
2C_0\de|I|=2C_0\eta\de.
\end{equation}

To bound the second term in \eqref{eq:Phi E minus U}, we observe that if $\alpha \in I'\setminus I \subset [\alpha_0-3\eta, \alpha_0+3\eta]$, then by \eqref{eq:phi nbhd G minus phi G} and \Cref{differenceset}(c),
\[
|\phi_\alpha(\nbhd{G}{\delta})|
\leq
2C_0\delta + |\phi_\alpha(G)|
\leq 
2C_0(\delta+\eta) 
\leq 
4C_0\eta.
\] 
Hence, by Fubini,
\begin{align*}
|\phi^*(E) \cap ((I'\setminus I) \times \R)|
\leq
|\phi^*(\nbhd{G}{\delta}) \cap ((I'\setminus I) \times \R)|
\leq
4C_0\eta|I'\setminus I|
=
8C_0\eta\de.
\end{align*}

To bound the third term in \eqref{eq:Phi E minus U}, we observe that if $\alpha\in [-N,N]\setminus I'$ then $|\phi_\alpha(E)|<\de$. Hence, by Fubini, \begin{align*}
|\phi^*(E) \cap (([-N,N]\setminus I') \times \R) | \le 2N\de. 
\end{align*}

Combining the three estimates above, we have shown that if $\delta < \min(\eta,\delta_0)$, then the corresponding set $E$ satisfies
\[
|(\phi^*(E) \setminus U) \cap ([-N, N] \times \mathbb{R})| \leq (2C_0\eta + 8C_0\eta + 2N)\delta.
\]
By choosing $\delta$ sufficiently small, this last expression is less than $\epsilon$, which proves (c).
\end{proof}

\subsection{Proof of Main Theorem}
Recall that we have already reduced the proof of our main theorem (Theorem \ref{thm:main v2}) to \Cref{lem:prescribedprojection with additional assumptions}.  We conclude the section by proving \Cref{lem:prescribedprojection with additional assumptions}.

\begin{proof}[Proof of \Cref{lem:prescribedprojection with additional assumptions}]
Let $\newA \subset \phi^*(\R^2)$ be a measurable set. By outer regularity and the fact that $\phi^*(\R^2)$ is open (\Cref{differenceset}), there exists a nested sequence of open sets $\newA_0 \supset \newA_1 \supset \cdots$ such that 
\begin{gather}
\label{eq:F in Fm in V0}
\newA \subset \newA_m \subset \phi^*(\R^2) \qquad\text{for all $m$} 
\\
\label{eq:diff to 0}
|\newA_m\setminus \newA|\to 0 \qquad\text{as } m\to\infty
\end{gather}  
For each $m \in \{0,1,2,\dots\}$, we construct a countable family of disks $\cB_m$ which satisfies the following:
\begin{enumerate}
    \item\label{item:B_m cover} $\phi^*(\cup \cB_m)\supset \newA_m$,
    \item\label{item:B_m efficient} $|(\phi^*(\cup\cB_m)\setminus \newA_m) \cap ([-m,m] \times \R)|\leq 2^{-m}$
    \item\label{item:nested B_m} For any $B\in \cB_{m-1}$, if $y\in\phi^*(B)\cap \newA_m$ then there exists $B'\in \cB_m$ with $B'\subset B$ and $y\in\phi^*(B')$.
\end{enumerate}
For convenience, we write $S_m$ for the vertical strip $[-m, m] \times \mathbb{R}$.

Since \eqref{eq:diff to 0} holds, condition \eqref{item:B_m efficient} implies $|(\phi^*(\cup\cB_m)\setminus F) \cap S_m|\to 0$ as $m\to\infty$.  Therefore, $\{\cB_m\}_m$ so constructed must satisfy the properties of  \Cref{definition: system-of-disks}, and hence completing this construction will complete the proof. The construction of $\{\cB_m\}_m$ is recursive.  

For the initial step, let $\cB_0$ be a countable collection of disks whose union is $\R^2$. Then \eqref{item:B_m cover} follows from \eqref{eq:F in Fm in V0}, \eqref{item:B_m efficient} is trivial since $|S_0| = 0$, and \eqref{item:nested B_m} is vacuously true.

Next, fix $m \in \N$ and suppose $\cB_0,\dots,\cB_{m-1}$ are countable families of open disks which satisfy properties \eqref{item:B_m cover}--\eqref{item:nested B_m}, and define $\cB_m$ as follows. Let $\{B_i\}_{i\in\N}$ be an enumeration of the disks in $\mathcal{B}_{m-1}$, and let $V_i=F_m\cap \phi^*(B_i)$. Since $V_i$ is open and $V_i \subset \phi^*(B_i)$, we can apply \Cref{base} to get a topological base $\mathcal{U}_i$ for $V_i$ satisfying the conclusion of that lemma.  Let $\{U_{i,j}\}_{j\in\N}\subset \mathcal{U}_i$ be a countable subset which still covers $V_i$.  For each $i,j\in\N$, let $E_{i,j}$ be a finite union of disks such that 
 \begin{gather}
\label{containment}
 E_{i,j}\subset B_i.\\
 \label{eq:Uij in phi Eij} 
 \phi^*(E_{i,j})\supset U_{i,j}
 \\
 \label{choice_of_epsilon_bound}
     |(\phi^*(E_{i,j})\setminus U_{i,j}) \cap S_m |\leq 2^{-(m+i+j)}.
 \end{gather}  
 Let $\cB_m$ be the (countable) collection of disks making up the sets $E_{i,j}$, as $i,j$ range over $\N$.  We note for future reference that we have
 \begin{equation}
     \label{decomposition}
 \cup\cB_m=\bigcup_{i,j}E_{i,j}
 \end{equation}
 and
 Now that we have defined $\cB_m$, it remains only to verify properties \eqref{item:B_m cover}--\eqref{item:nested B_m}.  To prove \eqref{item:B_m cover}, let $y\in \newA_m$.  For fixed $i\in\N$, the family $\{U_{i,j}\}_{j\in\N}$ covers $\newA_m\cap \phi^*(B_i)$ by construction.  Since $\{B_r\}_r$ is an enumeration of the disks in $\cB_{m-1}$, and since $\phi^*(\cup\cB_{m-1})\supset \newA_{m-1}\supset \newA_m$ by \eqref{item:B_m cover} applied to $\cB_{m-1}$, the family $\{U_{i,j}\}_{i,j\in\N}$ covers $\newA_m$.  This means $y\in U_{i,j}$ for some $i,j\in\N$, hence $y\in \phi^*(E_{i,j})\subset \phi^*(\cup\cB_m)$ by \Cref{base} and definition of $\cB_m$.  This proves \eqref{item:B_m cover}.

To prove \eqref{item:B_m efficient}, by \eqref{decomposition}, \eqref{choice_of_epsilon_bound}, and the fact that $U_{i,j}\subset \newA_m$, we have that
\[
|(\phi^*(\cup \cB_m) \setminus \newA_m) \cap S_m| \leq \sum_{i,j} |\phi^*(E_{i,j})\setminus U_{i,j} \cap S_m|\leq \sum_{i,j}2^{-(m+i+j)}=2^{-m}.
\] 

Finally, we prove \eqref{item:nested B_m}. Recall that $\cB_{m-1} = \{B_i\}_i$. Suppose $y\in \phi^*(B_i)\cap \newA_m$ for some $i$. Then $y\in U_{i,j}$ for some $j$ (since $\{U_{i,j}\}_{i,j}$ was constructed to cover that set), so by \eqref{eq:Uij in phi Eij}, it follows that $y\in \phi^*(E_{i,j})$.  Since $\cB_m$ is the collection of disks making up the sets $E_{i,j}$, it follows that there is a $B'\in \cB_m$ with $y\in \phi^*(B')$.  By \eqref{containment} we also have $B'\subset E_{i,j} \subset B_i$, completing the proof of \eqref{item:nested B_m}.
\end{proof}

\section{Geometric applications}
\label{applications}

\subsection{Efficient covering theorems}
Prescribed projections have an equivalent formulation in terms of covering arbitrary measurable sets in such a way that the excess has measure zero error.  More precisely, we make the following definition.
\begin{dfn}
Let $S\subset \R^2$.  We say that a parameterized family of sets $\cF=\{S_x:x\in X\}$ in $\R^2$ \textbf{yields efficient coverings} of subsets of $S$ if, for any measurable $F\subset S$, there exists $E\subset X$ such that
\[
\bigcup_{x\in E} S_x\supset F
\]
and
\[
\left|\left(\bigcup_{x\in E} S_x\right)\setminus F\right|=0.
\]
\end{dfn}

For $x\in \R^2$, define $A_x=\{\alpha\in \R: (\alpha,x)\in D$\}, and for $E\subset \R^2$, let $A(E)=\bigcup_{x\in E}A_x$.  Let $\mathcal{A}=A(\R^2)$.

\begin{thm}[Efficient covering theorem]
\label{efficientcoveringtheorem}
Let $\phi$ satisfy the hypotheses of \Cref{thm:main} (or, more generally, let $\phi$ be admissible), let
\[
X=\{x\in\R^2: A_x\neq\emp\},
\]
and for each $x\in X$ let
\[
S_x=\{(\alpha,\phi_\alpha(x)):\alpha\in A_x\}.
\]
The family $\cF=\{S_x:x\in X\}$ yields efficient coverings of subsets of $\bigcup_{\al\in \cA}\{\al\}\times \phi_\al(D_\al)$. 
\end{thm}
\begin{proof}
It will be convenient to work in the notation of \Cref{prescribedprojection} rather than \Cref{thm:main}.  We observe
\begin{align*}
\phi^*(E)&=\bigcup_{\alpha\in A(E)}\{\alpha\}\times \phi_\alpha(E) \\
&=\bigcup_{\alpha\in A(E)}\left(\bigcup_{x\in E\cap D_\alpha}\{(\alpha,\phi_\alpha(x))\}\right) \\
&=\bigcup_{x\in E}\left(\bigcup_{\alpha\in A_x}\{(\alpha,\phi_\alpha(x))\}\right) \\
&= \bigcup_{x\in E}S_x.
\end{align*}
The result then follows directly from Theorem \ref{prescribedprojection}.
\end{proof}

As a fundamental example, we note that this theorem unifies Davies' Theorem with the work of 
Chang, McDonald, and Taylor in \cite{CMT23}.  In that paper, the authors prove that if $\Gamma$ is the graph of a $C^1$ function on a compact interval with strictly monotone derivative, then the family of translates $\Gamma+x$ yields efficient coverings.  Although the proof techniques are heavily influenced by the proof of Davies' theorem on efficient coverings by lines, it does not include Davies' theorem as a special case since the family of lines does not consist of translates of any single line.  Our first two examples show that Theorem \ref{efficientcoveringtheorem} unifies\footnote{The result in \cite{CMT23} made the convention that domains were compact, whereas we make the convention that domains are open.  However, this should not be seen as a significant discrepancy.} these two results into one general framework.
\begin{ex}[Davies' Theorem \cite{Davies52}]
\label{davies}
Theorem \ref{efficientcoveringtheorem} applies to the family of lines.  More specifically, for any open interval $I$, we have efficient covering by lines with slopes in $I$.  Without loss of generality suppose $I$ is bounded, let $X=I\times\R$, and for $x\in X$ and $\alpha\in \R$ define
\[
\phi_\alpha(x)=x_1\alpha+x_2.
\]
In the notation of Theorem \ref{efficientcoveringtheorem}, we have
\[
S_x=\{(s,t)\in \R^2:t=x_1s+x_2\}.
\]
Clearly $\phi^*(\R^2)=\R^2$, so it suffices to show $\phi:\R\times X\to \R$ satisfies the hypotheses of Theorem \ref{thm:main}.  To show Condition (\ref{cond nsam}), we compute
\[
\nabla \phi_\alpha(x)=(\alpha,1).
\]
Thus, $\phi$ is non-stationary, and for any $x$ the map $\al\mapsto\theta_\al(x)$ is strictly decreasing when angles are defined relative to the positive horizontal axis.  Condition (\ref{cond ext}) is trivially satisfied, as $\phi$ and $\nabla_x\phi$ extend to continuous functions on $\R^3$ with $\nabla_x\phi \neq 0$, and Condition (\ref{cond conv}) (convexity of $D_\al = X = I \times \R$) is obvious.
\end{ex}
\begin{ex}[\protect{\cite[Theorem 1.1]{CMT23}}]
Let $f:[a,b]\to \R$ be a $C^1$ function with continuous and strictly monotone derivative, and let $\Ga^\circ$ be the graph of $f$ with the endpoints removed.  Theorem \ref{efficientcoveringtheorem} applies to the family of translates $S_x:=\Ga^\circ+x$.  Let 
\[
D=\{(\al,x):a+x_1<\al<b+x_1\},
\]
and define
\[
\phi_\alpha(x)=f(\alpha-x_1)+x_2.
\]
We then have
\begin{align*}
S_x&=\{(\alpha,f(\alpha-x_1)+x_2):\alpha\in (a,b)+x_1\} \\
&=\{(t,f(t)):a<t<b\}+x.
\end{align*}
Again, it is clear $\phi^*(\R^2)=\R^2$, so it suffices to verify the hypotheses of Theorem \ref{thm:main}.  We have
\[
\nabla \phi_\alpha(x)=(-f'(\alpha-x_1),1),
\]
so $\phi$ is non-stationary, and angle monotonicity follows from the assumption that $f'$ is strictly monotone.  This proves Condition (\ref{cond nsam}).  Condition (\ref{cond ext}) follows from the assumption that $f$ is continuously differentiable on $[a,b]$, and Condition (\ref{cond conv}) is clear.
\end{ex}

In addition to unifying previous results, Theorem \ref{efficientcoveringtheorem} provides efficient covering by many families of curves to which previous results are not applicable.  The following example concerns one such family, and serves to illustrate the flexibility of our theorem.

\begin{ex}
Let $\cF$ be the family of upper quarter circles with centers on the horizontal axis, of arbitrary radius.  The family $\cF$ yields efficient coverings of subsets of the upper half plane.  To prove this, we first observe that it is sufficient to consider the subfamily $\cF_\e$ of quarter circles with radius at least $\e$.  Since any subset of the upper half plane can be partitioned into countably many pieces where each is bounded away from the axis, the general result follows.  We can parameterize such circles by points $x=(x_1,x_2)\in \R\times (\e,\infty)$, where we think of $x_1$ as describing the center of the circle and $x_2$ as describing the radius.  More precisely, define
\[
D=\left\{(\al,x_1,x_2)\in \R\times\R\times (\e,\infty):x_1-\frac{x_2}{\sqrt{2}}<\al<x_1+\frac{x_2}{\sqrt{2}}\right\},
\]
and define projections $\phi:D\to (\frac{\e}{\sqrt{2}},\infty)$ by
\[
\phi_\alpha(x)=\sqrt{x_2^2-(\al-x_1)^2}.
\]
In the notation of Theorem \ref{efficientcoveringtheorem}, we then have
\[
S_x=\left\{(\alpha,\beta)\in \R\times (0,\infty): (\alpha-x_1)^2+\beta^2=x_2^2, |\alpha-x_1|<\frac{x_2}{\sqrt{2}}\right\},
\]
the upper quarter circle centered at $(x_1,0)$ of radius $x_2$.  We compute
\[
\nabla\phi_\al(x)=\frac{1}{\phi_\al(x)}(\al-x_1,x_2).
\]
It follows immediately that $\phi$ satisfies angle monotonicity, and the bound $x_2>\e$ implies that the function is non-stationary.  This proves Condition (\ref{cond nsam}).  By the bound $\phi_\al(x)>\e/\sqrt{2}$ on $D$, we get continuous extension to the boundary, proving Condition (\ref{cond ext}).  Finally, for each $\al$ the domain
\[
D_\al=\{(x_1,x_2)\in \R\times (\e,\infty):|x_1-\al|<\frac{x_2}{\sqrt{2}}\}
\]
is convex, proving (\ref{cond conv}).
\end{ex}

\subsection{Pinned distance sets}\label{pinned section}
\begin{dfn}
Let $E\subset \R^2$.  The \textbf{distance set} of $E$ is the set
\[
\Delta(E)=\{|x-y|:x,y\in E\}.
\]
For fixed $y\in \R^2$ (not necessarily in $E$), the \textbf{pinned distance set} at $y$ is
\[
\Delta_y(E)=\{|x-y|:x\in E\}.
\]
\end{dfn}
\begin{thm}[Prescribed pinned distance sets]
For each $\al\in\R$, let $F_\al\subset (0,\infty)$ be such that
\[
\bigcup_{\al\in\R}\{\al\}\times F_\al
\]
is measurable.  There exists a set $E\subset \R^2$ such that for almost every $\al$ we have
\[
\De_{(\al,0)}(E)\supset F_\al \hspace{.25in}\text{and}\hspace{.25in}|\De_{(\al,0)}(E)\setminus F_\al|=0.
\]
\end{thm}
\begin{proof}
For technical reasons, it will be simpler to work with the squared sets
\[
F_\al^2:=\{t^2:t\in F_\al\}.
\]
For $\epsilon > 0$, define
\[
F_{\al,\epsilon}^2:=\{t^2:t\in F_\al\} \cap (\epsilon^2, \infty).
\]
Define $\phi:\R\times (0, \epsilon/2) \times (0,\infty)\to (0,\infty)$ by
\[
\phi_\alpha(x_1,x_2)=|x-(\alpha,0)|^2=(x_1-\alpha)^2+x_2^2.
\]
With this definition, we have $\Delta_{(\alpha,0)}^2(E)=\phi_\alpha(E)$, where $\Delta^2$ denotes the set of squares of distances rather than distances themselves.  If Theorem \ref{thm:main} applies, then there exists $E$ such that $\Delta_{(\alpha,0)}^2(E)\supset F_{\alpha, \epsilon}^2$ and $|\Delta_{(\alpha,0)}^2(E)\setminus F_{\alpha, \epsilon}^2|=0$ for almost every $\alpha$.  Since $t\mapsto t^2$ is a bijection on $(0,\infty)$ which preserves the property of having measure zero, it is enough to verify the hypotheses of Theorem \ref{prescribedprojection}.  Clearly, $\phi_\alpha$ has range $(\epsilon^2/4,\infty)$ for each $\alpha$.  We have
\[
\nabla \phi_\al(x)=2(x_1-\al,x_2),
\]
so $\phi$ satisfies the angle monotonicity condition, is non-stationary (by the assumption $x_2>\epsilon/2$), and has derivatives extending continuously to the boundary.  Since convexity is clear, the result follows for $F_{\alpha, \epsilon}^2$. We can prove the result for $F_{\alpha}^2$ by taking a countable union over a sequence of values of $\epsilon$ that approach zero.
\end{proof}

\subsection{Radial projections}

\begin{dfn}
Let $\T$ denote the unit circle in $\R^2$.  For $a\in \R^2$, the \textbf{radial projection} to $a$ is the map 
\[
P_a:\R^2\setminus \{a\}\to \T
\]
defined by
\[
P_a(x)= \frac{x-a}{\|x-a\|}.
\]
\end{dfn}

\begin{thm}
For each $\alpha \in \mathbb{R}$, let $F_{\alpha} \subset \mathbb{T}$ be a set that does not contain $(1,0)$ or $(-1,0)$, and suppose that
\[\{(\alpha, t) : \alpha \in \mathbb{R} : (\cos t, \sin t) \in F_{\alpha}\}\]
is measurable. Suppose $y^- < 0 < y^+$ are real numbers, and suppose $U$ is an open neighborhood of the horizontal lines $\{(x,y^{-}) : x \in \mathbb{R}\}$ and $\{(x, y^+) : x \in \mathbb{R}\}$. Then there exists a set $E \subset U$ such that $P_{(\alpha,0)}(E) \supset F_{\alpha}$ and $|P_{\alpha, 0}(E) \setminus F_{\alpha}| = 0$ for almost every $\alpha \in \mathbb{R}$.
\end{thm}
\begin{rmk}
    It is interesting to note that a dual version of the prescribed radial projection theorem yields a covering theorem by line segments.  Since we already recover the stronger result for coverings by full lines above in Example \ref{davies}, we omit additional details. 
\end{rmk}
\begin{proof}
Observe that by taking a union of a set $E^+$ in the upper half-plane and $E^-$ in the lower half-plane, we may assume that the set $F_{\alpha}$ is contained in the open upper semicircle for each $\alpha$, and we may assume that $U$ is a neighborhood of the horizontal line $x_2 = y^+$. We make this assumption for the rest of the proof.

We define a function $\theta$ on the upper semicircle $\{(x_1, x_2) : x_1^2 + x_2^2 = 1, x_2 > 0\}$ by $\theta(x_1, x_2) = \cos^{-1}(x_1) \in (0, \pi)$. For each $\alpha$, let $G_{\alpha} = \theta(F_{\alpha})$, so that $G_{\alpha}$ is a subset of $(0, \pi)$ for each $\alpha$. Observe that the map $\theta$ is a bijection from the upper semicircle onto $(0, \pi)$ that preserves measure-zero sets. Therefore, if we define 
\[\phi_{\alpha}(x_1, x_2) := \cos^{-1} \left(\frac{x_1 - \alpha}{\sqrt{(x_1 - \alpha)^2 + x_2^2}} \right),\]
it is enough to find a set $E^+ \subset U$ such that $G_{\alpha} \subset \phi_{\alpha}(E^+)$ and $|\phi_{\alpha}(E^+) \setminus G_{\alpha}| = 0$ for almost all $\alpha \in \mathbb{R}$.

We would also like to avoid considering angles that are too close to $0$ or $\pi$. To do this, for each $N \in \mathbb{N}$, we define the set $G_{\alpha,N}$ as follows:
\[G_{\alpha,N} = \begin{cases}
G_{\alpha} \cap \left[\frac{1}{N}, \pi - \frac{1}{N} \right] & \text{ if $|\alpha| \leq N$} \\
\varnothing & \text{ if $|\alpha| > N$}.\end{cases}\]
We will show for each $N$ that there exists a set $E_N^+ \subset U$ such that $\phi_{\alpha}(E_N^+) \supset G_{\alpha, N}$ and $|\phi_{\alpha}(E_N^+) \setminus G_{\alpha, N}| = 0$ for almost every $\alpha \in \mathbb{R}$. Taking the union of the sets $E_N^+$ gives the desired set $E^+$.

Choose $0 < \delta_N < \frac{y^+}{100}$ such that the rectangle
\[R_N := \left(- N - y^+ \cot \left(\frac{1}{2N} \right), N + y^+\cot \left( \frac{1}{2N} \right) \right) \times (y^+ - \delta_N, y^+ + \delta_N) \subset U.\]
We claim that $G_{\alpha, N} \subset \phi_{\alpha}(R_N)$ for all $\alpha$. In fact, a simple calculation shows that if $\alpha \in [-N, N]$, then the image $\phi_{\alpha}(R_n)$ satisfies the containment

\[\left(\frac{1}{2N}, \pi - \frac{1}{2N} \right) \subset \phi_{\alpha}(R_N).\]
This containment shows that $\phi_{\alpha}(R_N)$ contains $G_{\alpha, N}$ for all $\alpha$ with $|\alpha| \leq N$.

We wish to apply Theorem 1.2 to $\phi_{\alpha}$ with $D_{\alpha} = R_N$ for each $\alpha$. Clearly $D_{\alpha}$ is convex for all $\alpha$. We now check that $\phi$ is nonstationary and satisfies the angle monotonicity condition. We need to compute the gradient of $\phi_{\alpha}$. To simplify notation, we define $\psi_{\alpha}(x) := |x - (\alpha,0)| = \sqrt{(x_1 - \alpha)^2 + x_2^2}$. With this notation, we have
\begin{equation}\label{eq:nablaphialpha}
\nabla \phi_{\alpha}(x) = \frac{x_2}{\psi_{\alpha}(x)^3 \sqrt{1 - \left( \frac{x_1 - \alpha}{\psi_{\alpha}(x)} \right)^2}} (-x_2, x_1 - \alpha).
\end{equation}
If $(x_1, x_2) \in \overline{R_N}$, the numbers $x_2$ and $\psi_{\alpha}(x)$ are at least $y^+ - \delta_n$, and $\psi_{\alpha}(x)$ is bounded above by a number that depends only on $N$, $\alpha$ and $y^+$. Therefore $\nabla \phi_{\alpha}$ is nonstationary on $D$, and both $\phi_{\alpha}$ and $\nabla \phi_{\alpha}$ can be continuously extended to $\overline{D}$. The angle monotonicity condition is also obviously satisfied since only the second coordinate of $\frac{\nabla \phi_{\alpha}(x)}{\norm{\nabla \phi_{\alpha}(x)}}$ depends on $\alpha$; in fact, it is obvious both from the geometry of the problem and the explicit expression \eqref{eq:nablaphialpha} that $\nabla \phi_{\alpha}$ will always be orthogonal to the vector $(x_1 - \alpha, x_2)$. Therefore, $\phi_{\alpha}$ meets all of the conditions required to apply Theorem \ref{thm:main}, guaranteeing the existence of the set $E_N^+$.
\end{proof}

\subsection{Nikodym-type sets}

\Cref{thm:main} can be used to construct Nikodym-type sets for general families of curves. This is done, for example, in the forthcoming paper \cite{ccgsyz}.

\bibliographystyle{plain}
\bibliography{refsprojections}

\end{document}